\documentclass[11pt]{amsart}

\usepackage{amsmath,amsfonts,amssymb,amsthm, bbold}
\usepackage{graphics,esint}
\usepackage{epsfig}
\usepackage{xcolor}

\usepackage{xfrac}
\usepackage{comment}
\usepackage{hyperref}

\newtheorem{definition}{Definition}[section]
\newtheorem{proposition}{Proposition}[section]
\newtheorem{theorem}{Theorem}[section]
\newtheorem{remark}{Remark}[section]
\newtheorem{lemma}{Lemma}[section]
\newtheorem{corollary}{Corollary}[section]

\newcommand{\R}{\mathbb{R}}
\newcommand{\Z}{\mathbb{Z}}

\newcommand{\supp}{\mathrm{supp }}

\def\S{\mathcal S}

\def\C{\mathbb C}

\def\SS{\mathcal S}
\def\T{\mathbb{T}^3}
\def\N{\mathbb N}
\def\eps{\varepsilon}

\def\mR{\mathring{R}}

\def\Id{\mathrm{Id}}
 
\def\div{\mathrm{div\,}}
\def\curl{\mathrm{curl\,}}
\def\tr{\mathrm{tr\, }}

\numberwithin{equation}{section}

\begin{document}

\title[Non-uniqueness and h-principle]{Non-uniqueness and h-principle for H\"older-continuous weak solutions of the Euler equations}

\author{Sara Daneri}
\address{Department Mathematik, Universit\"at Erlangen-N\"urnberg, 91058 Erlangen}
\email{daneri@math.fau.de}

\author{L\'aszl\'o Sz\'ekelyhidi Jr.} 
\address{Institut f\"ur Mathematik, Universit\"at Leipzig, 04009 Leipzig}
\email{szekelyhidi@math.uni-leipzig.de}

\maketitle

\begin{abstract}
In this paper we address the Cauchy problem for the incompressible Euler equations in the periodic setting. We prove that the set of H\"older $1\slash 5-\eps$ wild initial data is dense in $L^2$, where we call an initial datum wild if it admits infinitely many admissible H\"older $1\slash 5-\eps$ weak solutions. We also introduce a new set of stationary flows which we use as a perturbation profile instead of Beltrami flows in order to show that a general form of the h-principle applies to H\"older-continuous weak solutions of the Euler equations. Our result indicates that in a deterministic theory of 3D turbulence the Reynolds stress tensor can be arbitrary and need not satisfy any additional closure relation.
\end{abstract}

\section{Introduction}

In this paper we are concerned with the initial value problem for the incompressible Euler equations 
  \begin{equation}\label{E:e}
 \left\{\begin{aligned}
       &\partial_t v+\div (v\otimes v)+\nabla p=0\\
&\div v=0
        \end{aligned}\right.
\end{equation}
in the periodic setting $x\in \T=\R^3/\Z^3$. We study weak solutions $v$ which are H\"older continuous in space, i.e. such that 
\begin{equation}\label{e:Holder}
 |v(x,t)-v(y,t)|\leq C|x-y|^\theta,\quad \forall\, t\in [0,T]
\end{equation}
for some constant $C$ which is independent of $t$. Here $\theta\in (0,1)$ is the H\"older exponent. 
It is well known \cite{Scheffer93,Shnirelman:1997uz,DeLellis:2009jh,DeLellis:2008vc,DeLellis:2011vr} that in the class of weak solutions the initial value problem for \eqref{E:e} is not well-posed in a very strong sense: to a given initial value $v_0$ there may exist infinitely many weak solutions. 
Therefore a natural problem is to find appropriate admissibility criteria to be able to select a (physically relevant) unique solution. Unfortunately such a criterion is not presently known. Nevertheless, imposing the very mild and physically natural admissibility condition
\begin{equation}\label{e:admissibility}
\int_{\T}|v(x,t)|^2\,dx\leq  \int_{\T}|v_0(x)|^2\,dx
\end{equation}
already leads to a partial selection. Indeed, if $v_0$ is sufficiently smooth, the classical solution (which exists for some short time $[0,T)$ with constant energy) is in fact unique in the class of weak solutions which satisfy \eqref{e:admissibility} - this is known as weak-strong uniqueness, and holds even for measure-valued solutions \cite{Brenier:2011va}. On the other hand it is also known that if $v_0$ is not differentiable, there could exist infinitely many admissible weak solutions. For weak solutions which are merely bounded we refer to \cite{DeLellis:2008vc,SzekelyhidiJr:2012dc}, where it was shown that the class of such so-called ``wild'' initial data is dense in $L^2$. For weak solutions which are in addition H\"older continuous in space, the only available result so far is by the first author in \cite{Daneri:2014hb}: for any $\eps>0$ there exist infinitely many $1\slash10-\eps$ H\"older initial data which are wild in the sense that to any such initial data there exist infinitely many $1\slash 16-\eps$ H\"older solutions satisfying \eqref{e:admissibility}. 
Our aim in this paper is to continue the study of admissible H\"older-continuous weak solutions. In the following we state our main results.

\subsection{Density of wild initial data}\label{ss:density}

We start with a definition.

\begin{definition}Given a solenoidal vector field $v_0\in C^{\theta_0}(\T;\R^3)$, we say that $v_0$ is a \emph{wild initial datum in $C^{\theta}$} if there exist infinitely many weak solutions $v$ to \eqref{E:e} on $\T\times [0,1]$ with $v(x,0)=v_0(x)$ and satisfying \eqref{e:Holder} and \eqref{e:admissibility}. 
\end{definition}

With this definition the result in \cite{Daneri:2014hb} can be stated as follows: for any $\eps>0$ there exist infinitely many initial data in $C^{\sfrac{1}{10}-\eps}$ which are wild in $C^{\sfrac{1}{16}-\eps}$. Note that the loss in exponent from $1/10$ in the initial datum to $1/16$ for the solutions seems to be consistent with the example of shear flows in $C^{\theta}$ from \cite{Bardos:2009wm}, where also there is an instantaneous loss in regularity. However, it turns out that in our setting this loss in the exponent can be avoided. Our main result can be stated as follows:

\begin{theorem}[Density of wild initial data]\label{T:dense}
For any $\theta<1/5$ the set of divergence-free vector fields $v_0\in C^{\theta}(\T;\R^3)$ which are wild initial data in $C^{\theta}$ is a dense subset of the divergence-free vector fields in $L^2(\T;\R^3)$.
\end{theorem}

\subsection{h-principle}\label{ss:hprinciple}

Our second main result in this paper can be seen as an h-principle type result for equation \eqref{E:e}. In order to motivate the statement, recall the discussion in \cite{DeLellis:2011vr} concerning the Reynolds stress tensor and its relation to the notion of subsolutions. Let $v$  be a (deterministic weak or random turbulent) solution of \eqref{E:e} and consider a certain averaging process leading to the decomposition
$$
v=\overline{v}+w
$$
where $\overline{v}$ is the ``average'' and $w$ is the ``fluctuation''. The Euler equations \eqref{E:e} for $v$ transform into
\begin{equation}\label{e:averageE}
 \left\{\begin{aligned}
       &\partial_t \bar{v}+\div (\bar{v}\otimes \bar{v})+\nabla \bar{p}=-\div \bar{R}\\
&\div \bar{v}=0
        \end{aligned}\right.
\end{equation}
where 
\begin{equation}\label{e:Rstress}
\bar{R}=\overline{v\otimes v}-\overline{v}\otimes\overline{v}=\overline{w\otimes w}.
\end{equation}
Being an average of positive semidefinite tensors, it is easy to see that $\bar{R}$ is positive semidefinite. Accordingly, a subsolution is defined to 
be a triple $(\bar{v},\bar{p},\bar{R})$ where $\div\bar{v}=0$ such that \eqref{e:averageE} holds and $\bar{R}(x,t)\geq 0$ for almost every 
$(x,t)$. For a more precise definition, comparisons to the literature and further notions we refer to Section \ref{s:definitions}. 
In light of this interpretation of $\bar{R}$, it is natural to define the \emph{generalized energy tensor} of a subsolution $(\bar{v},\bar{p},\bar{R})$ to be the 
time-dependent tensor 
\begin{equation}\label{e:tensor}
\int_{\T}(\bar{v}\otimes \bar{v}+\bar{R})\,dx,
\end{equation}
and the associated \emph{generalized total energy} given by its trace (c.f.~\eqref{e:generalizedlocalenergy}):
$$
\overline{E}(t)=\frac{1}{2}\int_{\T}|\bar{v}|^2+\tr \bar{R}\,dx.
$$
In view of the previous discussion we call a subsolution $(\bar{v},\bar{p},\bar{R})$ \emph{admissible} if 
\begin{equation}\label{e:EnergyIneq}
\overline{E}(t)\leq \overline{E}(0)\quad\textrm{ for all $t>0$}.
\end{equation}
Observe that the system \eqref{e:averageE} is highly under-determined. 
An important question in the theory of turbulence is to obtain further restrictions on the tensor $\bar{R}$ in the form of constitutive (closure) relations. Thus an interesting question is whether there are additional constraints in the specific case where $\bar{R}$ arises -- in analogy with \eqref{e:Rstress} -- as a weak limit 
\begin{equation}\label{e:Reynolds}
\bar{R}=(\mathrm{w-lim}_{k\to\infty}v_k\otimes v_k)-v\otimes v,
\end{equation}
where $v_k\rightharpoonup v$ is a sequence of (admissible) H\"older continuous weak solutions. Indeed, weak convergence has long been considered as a useful tool to study ``deterministic turbulence'' \cite{Lax:1990eg}. 

Observe that if in \eqref{e:Reynolds} the sequence $v_k$ consists of Leray-Hopf weak solutions of the Navier-Stokes equations with vanishing viscosity $\nu_k\to 0$, then the limit $(\bar{v},\bar{p},\bar{R})$ will be an admissible subsolution\footnote{In this case $v$ arises as a weak* limit in the space $L^{\infty}(0,T;L^2(\T))$ and the limit in \eqref{e:Reynolds} is a weak* limit in the space of matrix-valued Radon measures. Nevertheless, $\overline{E}(t)$ is well-defined due to the energy inequality.}. This is well-known (see for instance \cite{Duchon:2000ti,Lions:1996vo}). Therefore the admissibility condition \eqref{e:EnergyIneq} together with the condition that $\bar{R}\geq 0$ gives rise to a natural class of subsolutions. 

It follows from \cite{DeLellis:2008vc,SzekelyhidiJr:2012dc} that no additional constraints on $\bar{R}$ exist for $L^\infty$ weak solutions, but for $C^\theta$ weak solutions this was open and in fact there were indications that constraints might exist \cite{Choffrut:2013gv,DeLellis:2012tz}. Our second main result shows that no additional constraints exist and that in fact \emph{any} positive definite tensor can arise as \eqref{e:Reynolds} from $C^{\sfrac15-\eps}$-weak solutions of Euler:

\begin{theorem}[h-principle]\label{t:hprinciple}
Let $(\bar{v},\bar{p},\bar{R})$ be a smooth solution of \eqref{e:averageE} on $\T\times[0,T]$ such that $\bar{R}(x,t)$ is positive definite for all $x,t$. Then there exists for any $\theta<1/5$ a sequence $(v_k,p_k)$ of weak solutions of \eqref{E:e} such that \eqref{e:Holder} holds, 
$$
v_k\overset{*}{\rightharpoonup} \bar{v}\quad\textrm{ and }\quad v_k\otimes v_k\overset{*}{\rightharpoonup} \bar{v}\otimes \bar{v}+\bar{R}\quad\textrm{ in }L^\infty
$$ 
uniformly in time and furthermore for all $t\in [0,T]$
$$
\int_{\T}v_k\otimes v_k\,dx=\int_{\T}(\bar{v}\otimes \bar{v}+\bar{R})\,dx.
$$
\end{theorem}

This result, which is new even for continuous solutions, can be interpreted as a precise analogue of the famous Nash-Kuiper theorem \cite{Nash:1954vt} on $C^1$ isometric embeddings. Indeed, in a nutshell the Nash-Kuiper theorem says that any smooth strictly short embedding of a closed $n$-dimensional Riemannian manifold $(M^n,g)\to \R^{n+1}$ can be approximated in the $C^0$-norm by $C^1$ isometric embeddings. Being strictly short amounts to the pointwise condition that the metric error $g-du\cdot du$ should be positive definite. In Theorem \ref{t:hprinciple} the corresponding condition is that the ``error'' $\bar{R}$ should be positive definite, so that $(\bar{v},\bar{p},\bar{R})$ plays the role of the strictly short map and the generalized energy tensor \eqref{e:tensor} plays the role of the given metric $g$. As the Euler equations \eqref{E:e} are in divergence-form, in terms of differentiability $C^1$ isometric maps correspond naturally to $C^0$ weak solutions of \eqref{E:e}. Thus, in analogy with the Nash-Kuiper result, Theorem \ref{t:hprinciple} says that any smooth strict subsolution can be weakly approximated by $C^0$ weak solutions with prescribed energy. For more information on the connection between the Nash-Kuiper iteration and the Euler equations we refer to the surveys \cite{DeLellis:2011vr,DeLellis:2015uq,SzekelyhidiJr:V8n9YaGH} and the 
lecture notes \cite{SzekelyhidiJr:2016wc}.

\subsection{Approximating arbitrary background flows}\label{ss:approximation}

Our Theorem \ref{t:hprinciple} can also be seen as a contribution towards understanding the approximability of arbitrary smooth solenoidal background flows $\bar{v}$ by  (H\"older-)continuous weak solutions of the Euler equations. Indeed, it is easy to see (for instance using the operator $\mathcal{R}$ from \eqref{d:DefOpR} below) that to any smooth divergence-free $\bar{v}=\bar{v}(x,t)$ with spatial mean zero (for all time $t$) there exist smooth $\bar{p},\mathring{R}$ such that $(\bar{v},\bar{p},\mathring{R})$ is a smooth solution of \eqref{e:averageE}, and $\mathring{R}$ is a symmetric and traceless. Moreover, setting
\begin{equation}\label{e:approxRrho}
\bar{R}(x,t)=\frac{1}{3}\rho(t)\Id+\mathring{R}(x,t),
\end{equation}
with an appropriate choice\footnote{It suffices that 
$$
\rho(t)>3\sup_{x\in \T}\sup_{\xi\in S^2}\bigl(-\mathring{R}(x,t)\xi\cdot\xi\bigr).
$$} of $\rho(t)>0$ we can ensure that $\bar{R}$ is positive definite for all $x,t$ so that $(\bar{v},\bar{p},\bar{R})$ satisfies the conditions of Theorem \ref{t:hprinciple}, and thence obtain a sequence of H\"older-continuous weak solutions $v_k$ with $v_k\overset{*}{\rightharpoonup} \bar{v}$.
Such a result has been obtained in \cite[Theorem A.1]{Isett:2014vl} in the non-periodic setting. However, in \cite[Theorem A.1]{Isett:2014vl} there is no associated control of the energy - although energy control can be easily obtained by adapting the proof of Theorem 1.1 in \cite{Isett:2014vl}, the energy obtained will be very large compared with the energy of $\bar{v}$.  

As explained in the introduction, there is a substantial qualitative difference between admissible and non-admissible weak solutions. Moreover, from Theorem \ref{t:hprinciple} we also obtain control of the energy in the form that for all $t$
\begin{equation*}
\fint_{\T}|v_k|^2\,dx=\fint_{\T}|\bar{v}|^2+\tr\bar{R}\,dx=\fint_{\T}|\bar{v}|^2\,dx+\rho(t).
\end{equation*}
On the other hand it is also quite clear that $\bar{p},\bar{R}$ are not uniquely determined, 
hence a natural problem is to estimate the minimum energy level required for admissible weak solutions $v_k$ that can approximate $\bar{v}$. In light of the identity 
$$
\lim_{k\to\infty}\int_{\T}|v_k-\bar{v}|^2\,dx=\lim_{k\to\infty}\int_{\T}|v_k|^2-|\bar{v}|^2\,dx
$$
this is related to the \emph{distance in the strong $L^2$ topology between the vectorfield $\bar{v}$ and the set of weak solutions}. 

In \cite[Theorem 1.1]{Choffrut:2013gv} a weaker version of Theorem \ref{t:hprinciple} was obtained, where a key difference is the more restrictive condition on $\bar{R}$ (see also Remark \ref{r:strongsub} below). The precise condition, stated in terms of the decomposition in \eqref{e:approxRrho}, is the following:
$$
\frac{1}{6}\rho(t)\Id+\mathring{R}(x,t)\textrm{ is positive definite for all $x,t$.}
$$
Therefore even the choice of $\bar{R}$ in \eqref{e:approxRrho} in combination with Theorem \ref{t:hprinciple} is an improvement. Moreover, the full generality of Theorem \ref{t:hprinciple} gives an implicit characterization of the smallest possible energy level in terms of the possible choices of $(\bar{p},\bar{R})$ which is essentially optimal in light of the discussion in Section \ref{ss:hprinciple} above.

As a final remark note the following special case of Theorem 1.2: if $\bar{v}$ is a smooth exact solution of the Euler equations, then $\bar{v}$ can be weakly approximated by H\"older-continuous weak solutions $v_k$ with $0<\sup_t\int|v_k|^2-|v|^2\,dx<\delta$ for arbitrary $\delta>0$. This result has been proved in \cite[Theorem 1.2]{Isett:2014vl}.

\subsection{Comments on the proofs}

One of the novelties of this paper is that we introduce a new class of perturbations which we call \emph{Mikado flows}. These serve the purpose of replacing the Beltrami flows in order to generate arbitrary Reynolds stresses. As in \cite{DeLellis:2013im,DeLellis:2012tz} and later improved in \cite{Isett:2013ux,Buckmaster:2014ty}, the basic iteration scheme consists of adding at each step a fast oscillating perturbation with profile
$$
W=W(R,\xi),
$$
where $\T\ni \xi\mapsto W(R,\xi)$ is a periodic stationary solution of Euler with the property that
\begin{equation}\label{e:WoW}
\langle W\otimes W\rangle =R.	
\end{equation}
(here $\langle\cdot\rangle$ denotes spatial average on $\T$). For a detailed exposition of these ideas we refer to the lecture notes \cite{SzekelyhidiJr:2016wc}. In all previous works Beltrami flows were used as the family of stationary flows for $W$. However Beltrami flows are not sufficiently rich to allow any positive definite matrix $R$ in \eqref{e:WoW} (see \cite{Choffrut:2013gv}), and consequently $R$ in the iteration was restricted to be in a neighbourhood of the identity matrix (c.f.~Lemma \ref{l:split} and Definition \ref{D:strosubsol} below). This restriction is avoided with our Mikado flows, which are indeed a sufficiently rich family. 

A second novel feature in our paper is that we identify a new class of \emph{adapted subsolutions}, where vanishing of the Reynolds tensor is allowed but is coupled to the blow-up of the $C^1$ norm at a specific rate - see Definition \ref{D:adapt}. A somewhat different notion of admissible subsolution was previously introduced in \cite{Daneri:2014hb}, but there the subsolution depends on a specific iteration scheme and cannot be used in our setting for Theorem \ref{T:dense}.

\section{Preliminaries}
\label{S:pre}
\subsection{Function spaces}
In the following $m=0,1,2,\dots,\,\alpha\in(0,1),$ and $\beta$ is a multi-index. We introduce the usual (spatial) H\"older norms as follows. First of all, for $f:\T\to\R^3$, the supremum norm is denoted by $\|f\|_0=\underset{\T}{\sup}f(x)$. The H\"older seminorms are defined as 
\begin{align}
 [f]_m&=\underset{|\beta|=m}{\max}\|D^\beta f\|_0,\\
[f]_{m+\alpha}&=\underset{|\beta|=m}{\max}\underset{x\neq y,t}{\sup}\frac{D^\beta f(x,t)-D^\beta f(y,t)}{|x-y|^\alpha},
\end{align}
where $D^\beta$ are space derivatives. The H\"older norms are then given by
\begin{align}
 \|f\|_m=\underset{j=0}{\overset{m}{\sum}}[f]_j\\
\|f\|_{m+\alpha}=\|f\|_m+[f]_{m+\alpha}.
\end{align}
Recall the standard interpolation inequalities
\begin{equation}
 [f]_s\leq C\|f\|_0^{1-\frac{s}{r}}[f]_r^{\frac{s}{r}}.
\end{equation}
Next, we recall that $H^1_0(\T)$ is the usual Sobolev space of periodic functions $f:\T\to\R^3$ with average zero $\int_{\T}f(x)\,dx=0$, and $H^{-1}(\T)$ denotes its dual space, with norm
$$
\|f\|_{H^{-1}}=\sup_{\|\varphi\|_{H_0^1}\leq 1}\int_{\T}f\varphi\,dx.
$$
 
In this paper we will consider the spatial H\"older norms of time-dependent functions $f:\T\times[0,T]\to\R^3$. These will be denoted 
as $[f(t)]_m$ and $\|f(t)\|_m$. If the estimates hold uniformly in time, the explicit time dependence will be omitted.

\subsection{Elliptic operators, Schauder estimates and stationary phase lemma}

In this section we recall the operator $\mathcal{R}$ from Section 4.5 in \cite{DeLellis:2013im}, which is used as a right inverse for the divergence operator on matrices.
\begin{definition}\label{d:DefOpR}
 Let $v\in C^{\infty}(\T;\R^3)$ be a smooth vector field. We then define $\mathcal Rv$ to be the matrix valued periodic function 
\[
 \mathcal Rv:=\frac{1}{4}(\nabla \mathcal Pu+(\nabla \mathcal Pu)^T)+\frac{3}{4}(\nabla u+(\nabla u)^T)-\frac{1}{2}(\div u)\Id,
\]
where $u\in C^\infty(\T;\R^3)$ is the solution of
\[
 \triangle u=v-\fint_{\T} v\quad \text{in $\T$}
\]
with $\int u=0$ and $\mathcal{P}$ is the Leray projection onto divergence-free fields with zero average.
\end{definition}

\begin{lemma}
 [$\mathcal R=\div^{-1}$]
For any $v\in C^\infty(\T;\R^3)$ we have
\begin{align}
 &\mathcal Rv(x) \text{ is a symmetric trace free matrix for each $x\in\T$};\label{e:Rproperty1}\\
&\div \mathcal Rv=v-\fint_{\T} v.\label{e:Rproperty2}
\end{align}

\end{lemma}

Moreover, for nonlinear phase functions we have the following version of the stationary phase lemma from \cite{DeLellis:2013im}, whose proof is a simple modification of the proof of Proposition 5.2 in \cite{DeLellis:2013im} and standard Schauder estimates.

\begin{lemma}[Stationary phase lemma]\label{L:stat}
Let $\varphi\in C^\infty(\R^3)$ such that there exist $c_0>1$ and $k_0\in \Z^3$ such that, for all $x\in\R^3$ and $k\in \Z^3$
\begin{align}
 c_0^{-1}&\leq|\nabla \varphi(x)|\leq c_0,\label{E:varphiC0}\\
\varphi(x+2\pi k)&=\varphi(x)+2\pi k\cdot k_0 .
\end{align}
Then, 
\begin{itemize}
 \item [(i)] for any $a\in C^\infty(\T)$ and $N\in\N$
\begin{equation}\label{E:stati}
 \Big|\int_{\T}a(x)e^{i\lambda\varphi(x)}\,dx\Big|\leq C\frac{[a]_N+\|a\|_0[\nabla\varphi]_{N}}{\lambda^N}
\end{equation}\\
\item[(ii)] the solution $u\in C^\infty(\T)$ of
\begin{equation}
 \left\{\begin{aligned}
         \triangle u&=ae^{i\lambda\varphi}-\fint_{\T} ae^{i\lambda \varphi}\\
         \fint_{\T} u&=0
        \end{aligned}\right.
\end{equation}
satisfies 
\begin{align}
 [\nabla u]_\alpha&\leq C\Big\{\frac{1}{\lambda ^{1-\alpha}}\|a\|_0+\frac{1}{\lambda^{N-\alpha}}(\|a\|_0[\nabla\varphi]_N+[a]_N)\notag\\
&+\frac{1}{\lambda^{N}}(\|a\|_0[\nabla\varphi]_{N+\alpha}+[a]_{N+\alpha})\Big\}.\label{E:statii}
\end{align}

\end{itemize}
In particular, 
\begin{align}
 [\mathcal R(ae^{i\lambda\varphi})]_\alpha&\leq C\Big\{\frac{1}{\lambda ^{1-\alpha}}\|a\|_0+\frac{1}{\lambda^{N-\alpha}}(\|a\|_0[\nabla\varphi]_N+[a]_N)\notag\\
&+\frac{1}{\lambda^{N}}(\|a\|_0[\nabla\varphi]_{N+\alpha}+[a]_{N+\alpha})\Big\}.\label{E:statiii}
\end{align}
The constant $C$ depends on $c_0$, $N\in \N$ and $\alpha\in (0,1)$.
\end{lemma}

\begin{proof}
 Let $a_0=a$ and 
\[
a_{n}=-\div \Big(a_{n-1}\frac{\nabla\varphi}{|\nabla\varphi|^2}\Big),\quad\forall\,n=1,\dots, N.
\]
It follows by induction on $N$ that
\begin{equation}\label{E:a0formula}
 a_0e^{i\lambda\varphi}=\sum_{n=0}^{N-1}\div\Big(\frac{a_n\nabla\varphi}{|\nabla\varphi|^2}e^{i\lambda\varphi}\Big)\frac{1}{(i\lambda)^{n+1}}+\frac{1}{(i\lambda)^N} a_Ne^{i\lambda\varphi}\,.
\end{equation}
Moreover, again by induction, for all $j=0,\dots,N$ we have
\begin{align}
\Big[a_N\frac{\nabla\varphi}{|\nabla\varphi|^2}\Big]_s&\leq C(j,N,c_0)([a_j]_{N-j+s}+\|a_j\|_0[\nabla\varphi]_{N-j+s}).\label{E:a_N2}
\end{align}
Then, (i) immediately follows.

According to standard Schauder estimates,
 \begin{equation}
  [\nabla u]_\alpha\leq C \frac{1}{\lambda}\sum_{n=0}^{N-1}\frac{1}{\lambda^{n}}\Big[\frac{a_n\nabla\varphi}{|\nabla\varphi|^2}e^{i\lambda\varphi}\Big]_\alpha+\frac{1}{\lambda^N} \Big[a_Ne^{i\lambda\varphi}\Big]_\alpha+\Big|\fint_{\T}ae^{i\lambda\varphi}\Big|
 \end{equation}
which using $(i)$ and \eqref{E:a_N2} gives $(ii)$, as in the proof of Proposition 5.2 in \cite{DeLellis:2013im}.

\end{proof}

\subsection{Mikado flows}
\label{Ss:mikado}

In this section we introduce a new family of periodic stationary solutions of the Euler equations whose spatial averages will be used to absorb the Reynolds stress of general subsolutions (see Section \ref{s:definitions}). In the following $ \SS^{3\times3}_+$ denotes the set of positive definite symmetric $3\times 3$ matrices. 

\begin{lemma}\label{l:Mikado}For any compact subset $\mathcal N\subset\subset \SS^{3\times3}_+$
there exists a smooth vector field 
$$
W:\mathcal N\times \T\to \R^3,\quad i=1,2
$$
such that, for every $R\in\mathcal N$ 
\begin{equation}\label{e:Mikado}
\left\{\begin{aligned}
\div_\xi(W(R,\xi)\otimes W(R,\xi))&=0,\\
\div_\xi W(R,\xi)&=0,
\end{aligned}\right.
\end{equation}
and
\begin{eqnarray}
	\fint_{\T} W(R,\xi)\,d\xi&=&0,\label{e:MikadoW}\\
    \fint_{\T} W(R,\xi)\otimes W(R,\xi)\,d\xi&=&R.\label{e:MikadoWW}
\end{eqnarray}
\end{lemma}

The first step in the proof of Lemma \ref{l:Mikado} is the following geometric lemma from \cite{Nash:1954vt} Lemma 1, see also \cite{SzekelyhidiJr:2014tu} Lemma 1.9:
\begin{lemma}\label{L:geomNash}
For any compact subset $\mathcal N\subset\subset \SS^{3\times3}_+$ there exists $\lambda_0\geq1$ and smooth functions $\Gamma_k\in C^\infty(\mathcal N;[0,1])$ for any $k\in\Z^3$ with $|k|\leq\lambda_0$ such that 
\begin{equation}\label{E:Rrepr}
R=\underset{k\in\Z^3,|k|\leq\lambda_0}{\sum}\Gamma_k^2(R)k\otimes k\quad \textrm{ for all }R\in\mathcal{N}.
\end{equation}
\end{lemma}

\begin{proof}[Proof of Lemma \ref{l:Mikado}]
Then, let us look for our vector field $W(R,\cdot)$ among the vector fields of the form
\begin{equation}\label{E:wdense}
W(R,\xi)=\underset{k\in\Z^3,|k|\leq\lambda_0}{\sum}\Gamma_k(R)\psi_k(\xi)k
\end{equation}
where $\psi_k(\xi)=g_k(\mathrm{dist}(\xi, \ell_{k,{p_k}}))$ with $g_k\in C^\infty_c([0,r_k))$, $r_k>0$, where $\ell_{k,{p_k}}$ is the $\T$-periodic extension of the line $\{p_k+tk:\,t\in\R\}$ passing through $p_k$ in direction $k$. Since there are only a finite number of such lines, we may choose $p_k$ and $r_k>0$ in such a way that 
\begin{equation}\label{e:disjoint}
\mathrm{supp }\,\psi_i\cap\mathrm{supp }\,\psi_j=\emptyset\qquad\text{for all $i\neq j$.}
\end{equation}
Thus $W$ consists of a finite collection of disjoint straight tubes such that in each tube $W$ is a straight pipe flow and outside the tubes $W=0$.
In particular $W$ satisfies the stationary pressureless Euler equations \eqref{e:Mikado}.
Furthermore, the profile functions $g_k$ will be chosen so that $\int_{\T}\psi_k(\xi)\,d\xi=0$ and
$$
\fint_{\T} \psi_k^2(\xi)\,d\xi=1\quad\textrm{ for all $k$}.
$$
Then \eqref{e:MikadoW} is easily satisfied, and because of \eqref{e:disjoint} we also have
$$
\fint_{\T}W\otimes W\,d\xi=\sum_{k}\fint_{\T}\Gamma_k^2(R)\psi_k^2(\xi)k\otimes k\,d\xi=\sum_{k}\Gamma_k^2(R)k\otimes k=R.
$$
Therefore \eqref{e:MikadoWW} is satisfied.
\end{proof}

\section{Subsolutions}\label{s:definitions}

In this section we introduce several notions of subsolutions. 
 
\begin{definition}[Subsolution]\label{d:subsolution} 
A \emph{subsolution} is a triple 
$$
(v,p,R):\T\times(0,T)\to \R^3\times\R\times\SS^{3\times3}
$$
such that $v\in L^2_{loc}$, $R\in L^1_{loc}$, $p$ is a distribution, the equations
\begin{align}\label{E:er1}
\partial_t v+\div (v\otimes v)+\nabla p&=-\div R\notag\\
\div v&=0
\end{align}
hold in the sense of distributions in $\T\times(0,T)$, and moreover $R\geq 0$ a.e. 
We call a subsolution \emph{strict} if $R>0$ a.e.
\end{definition}

In the above definition $R\geq 0$ a.e. means $R(x,t)$ is positive semi-definite for almost every $(x,t)$.

\begin{remark}
Note that this definition agrees with the notion of subsolution introduced in \cite{DeLellis:2011vr} (Definition 2.3) in the following sense. Recall that a triple 
$$
(v,u,q):\T\times[0,T]\to \R^3\times\SS^{3\times3}_0\times\R
$$
is called a subsolution with respect to energy density $\bar{e}=\bar{e}(x,t)\geq 0$ if 
\begin{align*}
\partial_t v+\div u+\nabla q&=0\\
\div v&=0
\end{align*}
and 
$$
v\otimes v-u\leq \frac{2}{3}\bar{e}\,\Id\quad{a.e.}
$$
Given such a triple $(v,u,q)$ we obtain a subsolution in the sense of Definition \ref{d:subsolution}
by setting
$$
R=\frac{2}{3}\bar{e}\,\Id-v\otimes v+u,\quad p=q-\frac{2}{3}\bar{e}.
$$
Conversely, any subsolution $(v,p,R)$ defines a subsolution $(v,u,q)$ with energy density
\begin{equation}\label{e:generalizedlocalenergy}
\bar{e}=\frac{1}{2}\left(\tr R+|v|^2\right)=\frac{1}{2}\tr(R+v\otimes v)
\end{equation}
by setting 
$$
u=R-\frac{2}{3}\bar{e}\,\Id+v\otimes v,\quad q=p+\frac{2}{3}\bar{e}.
$$
\end{remark}

Next, we look at two more notions of subsolutions.
In the following, we denote by $\mR$ the traceless part of the tensor $R$, as in \cite{DeLellis:2013im,DeLellis:2012tz,Buckmaster:2014ty}.
\begin{definition}[Strong subsolution]\label{D:strosubsol}
A \emph{strong subsolution} is a subsolution $(v,p,R)$ such that in addition $\tr R$ is a function of $t$ only, and if \begin{equation}\label{e:Rstr3}
\rho(t):=\tfrac{1}{3}\tr R(t),
\end{equation}
then
\begin{equation}\label{e:Rdeltaest}
\Bigl|\mathring{R}(x,t)\Bigr|\leq r_0\rho(t)\textrm{ for all $(x,t)$,}
\end{equation}
where $0<r_0<1$ is the radius in Lemma 3.2 in \cite{DeLellis:2013im} (Lemma \ref{l:split} below).
\end{definition}

Observe that, writing
$$
R(x,t)=\rho(t)\Id+\mathring{R}(x,t)
$$
with $\tr\mathring{R}=0$, inequality \eqref{e:Rdeltaest} is equivalent to
$$
\Big|\frac{R(x,t)}{\rho(t)}-\Id\Big|\leq r_0
$$ 
provided $\rho(t)>0$.
Note also that, if a symmeric tensor $R$ satisfies this inequality, then (since $r_0\leq 1$)
\begin{equation}\label{e:Rvsrho}
|R|\leq 2\rho.	
\end{equation}

\begin{remark}\label{r:strongsub}
The notion of strong subsolution and in particular condition \eqref{e:Rdeltaest} is motivated by the constructions in \cite{DeLellis:2013im,DeLellis:2012tz,Buckmaster:2014ty}, based on Beltrami flows - see Proposition \ref{p:basic} below. Furthermore, although not equivalent, our definition of strong subsolution is also closely related to the definition given in \cite{Choffrut:2013gv}. Also in that paper the motivation was to have a notion of subsolution to which an iteration scheme based on Beltrami flows can be applied. 
\end{remark}

\bigskip

The definition of strong subsolution involves the radius $r_0$, which appears in the geometric decomposition in Lemma 3.2 from \cite{DeLellis:2013im}. In the rest of this paper we will be repeatedly applying perturbations to strong subsolutions, under which the inequality \eqref{e:Rdeltaest} is not stable. To circumvent this issue, we will use a collection of smaller radii $r_3<r_2<r_1<r_0$, whose numerical value is not important, we will only require that
\begin{equation}\label{E:r3r2r1}
r_1\leq\frac14r_0,\quad r_2<\frac{r_1}{2},\quad r_3<\frac{r_2}{2}.
\end{equation}

Our first result allows the weak approximation of arbitrary smooth strict subsolutions with strong subsolutions.

\begin{proposition}\label{p:stristro}
Let $({v},{p},{R})$ be a smooth strict subsolution on $\T\times [0,T]$. Then there exists $\delta_0>0$ such that, for any $0<\delta<\delta_0$ and any $\sigma>0$ there exists a smooth strong subsolution 
$(\bar v,\bar p,\bar R)$ such that, for all $t\in [0,T]$
\begin{equation}\label{E:H-1stristro}
\|v-\bar{v}\|_{H^{-1}}<\sigma,
\end{equation}
\begin{equation}\label{e:H-1stristo2}
	\|v\otimes v+R-\bar{v}\otimes\bar{v}-\bar{R}\|_{H^{-1}}<\sigma,
\end{equation}
\begin{equation}\label{E:stristro1}
\tfrac{3}{4}\delta\leq \tfrac{1}{3}\tr\bar{R}(t)\leq \tfrac{5}{4}\delta,
\end{equation}
\begin{equation}\label{E:stristro2}
\int_{\T}(v\otimes v+R)\,dx=\int_{\T}(\bar{v}\otimes\bar{v}+\bar{R})\,dx
\end{equation}
and for all $(x,t)\in\T\times[0,T]$
\begin{equation}\label{E:stristro3}
\left|\mathring{\bar R}(x,t)\right|\leq \frac{r_3}{3}\tr \bar R(t).
\end{equation}
\end{proposition}
 
\bigskip

In our third notion of subsolution, the possible vanishing of the Reynolds tensor at the initial time is allowed at the expense of the blow-up of $C^1$ norms as specific rates, which are consistent with H\"older-continuity at exponent $1/5-\eps$. 

\begin{definition}[Adapted subsolution]\label{D:adapt}
Given $\theta\in (0,1/5)$ we call a triple $(\bar{v},\bar{p},\bar{R})$ a $C^{\theta}$-adapted subsolution on $[0,T]$ if
$$
(\bar{v},\bar{p},\bar{R})\in C^{\infty}(\T\times (0,T])\cap C(\T\times [0,T])
$$
is a strong subsolution with initial data
$$
\bar{v}(\cdot,0)\in C^{\theta}(\T),\, \bar{p}(\cdot,0)\in C^{2\theta}(\T),\,\bar{R}(\cdot,0)\equiv 0,
$$
and, with $\bar{\rho}(t)=\frac{1}{3}\tr\bar R(t)$, we have for all $t>0$ and $x\in\T$
$$
\left|\mathring {\bar R}(x,t)\right|\leq r_2\bar{\rho}(t)
$$
and there exists a constant $M>1$ and $\eps>0$ with $\frac{1}{5+2\eps}>\theta$ such that for all $t>0$ we have $\bar{\rho}(t)>0$ and
\begin{equation}\label{e:adapted-est}
\begin{split}
[\bar{v}(t)]_1&\leq M\bar{\rho}(t)^{-(2+\eps)}\,,\quad [\bar{p}(t)]_1\leq M\bar{\rho}(t)^{-(3/2+\eps)}\,,\\
[\bar{R}(t)]_1&\leq M\bar{\rho}(t)^{-(3/2+\eps)}\,,\quad\|(\partial_t+\bar{v}\cdot\nabla)\bar{R}(t)\|_0\leq M\bar{\rho}(t)^{-(1+\eps)}\,.
\end{split}
\end{equation}
\end{definition}

Recall that in this paper we use the notation $[v(t)]_1$ to denote the spatial $C^1$ seminorm of the function $v=v(x,t)$ at time $t$.

\begin{proposition}\label{p:stradapt}
Let $(v_0,p_0,R_0)\in C^{\infty}(\T\times[0,T])$ be a smooth strong subsolution such that 
$$
\tfrac{3}{4}\delta\leq \tfrac{1}{3}\tr\,R_0(t)\leq\tfrac{5}{4}\delta\quad\textrm{ for all }t\in [0,T] 
$$
for some $\delta>0$ and, for all $(x,t)\in\T\times [0,T]$
\[
\left|\mR_0(x,t)\right|\leq \frac{r_3}{3}\tr\,R_0(t).
\] 
Then, for any $\theta<1/5$ and $\sigma>0$ there exists a $C^{\theta}$-adapted subsolution
$(\bar{v},\bar{p},\bar{R})$ such that, for all $t\in[0,T]$
$$
\int_{\T}\bar{v}\otimes \bar{v}+\bar{R}\,dx=\int_{\T}v_0\otimes v_0+R_0\,dx
$$
and
$$
\|\bar{v}-v_0\|_{H^{-1}}<\sigma,
$$
\begin{equation*}
	\|\bar{v}\otimes \bar{v}+\bar{R}-v_0\otimes v_0-R_0\|_{H^{-1}}<\sigma.
\end{equation*}
\end{proposition}

As an immediate consequence of Propositions \ref{p:stristro} and \ref{p:stradapt} we obtain the following statement:
\begin{corollary}\label{c:striadapt}
  Let $(\tilde{v},\tilde{p},\tilde{R})$ be a smooth strict subsolution on $\T\times [0,T]$. Then, for any $\theta<1/5$ and $\sigma>0$ there exists a $C^{\theta}$-adapted subsolution
$(\bar{v},\bar{p},\bar{R})$ such that, for all $t\in[0,T]$
$$
\int_{\T}\bar{v}\otimes \bar{v}+\bar{R}\,dx=\int_{\T}\tilde v\otimes \tilde v+\tilde R\,dx
$$
and
$$
\|\bar{v}-\tilde v\|_{H^{-1}}< \sigma,
$$
\begin{equation*}
	\|\bar{v}\otimes \bar{v}+\bar{R}-\tilde v\otimes \tilde v-\tilde R\|_{H^{-1}}<\sigma.
\end{equation*}
\end{corollary}

Our final result in this section shows that $C^\theta$-adapted subsolutions can be used for the initial value problem for $1/5$-H\"older weak solutions:
\begin{proposition}\label{P:adaptsol}
Let $(\bar{v},\bar{p},\bar{R})$ be a $C^\theta$-adapted subsolution with $\theta<1\slash5$. 
Then, for any $\sigma>0$ there exists a continuous weak solution $(v,p)$ of \eqref{E:e} with initial data
\[
v(\cdot,0)=\bar{v}(\cdot,0),
\]
such that for all $t\in [0,T]$
\begin{equation}\label{e:Ctheta-est}
	|v(x,t)-v(y,t)|\leq C|x-y|^{\theta}\quad\textrm{ for all $x,y\in\T$}
\end{equation}
for some constant $C>0$, 
\begin{equation}
 \int_{\T} v\otimes v\,dx=\int_{\T} \bar{v}\otimes \bar{v}+\bar{R}\,dx
\end{equation}
and
\begin{equation}\label{E:h-12}
\|\bar v-v\|_{H^{-1}}< \sigma,
\end{equation}
\begin{equation*}
	\|\bar{v}\otimes \bar{v}+\bar{R}-v\otimes v\|_{H^{-1}}<\sigma.
\end{equation*}
\end{proposition}

As an immediate consequence of Proposition \ref{P:adaptsol} we obtain the following criterion for wild initial data:

\begin{corollary}\label{c:wild}
	Let $w\in C^{\theta}(\T)$ be a divergence-free vectorfield for some $\theta<1/5$. If there exists a $C^\theta$-adapted subsolution
	$(\bar{v},\bar{p},\bar{R})$ such that $\bar{v}(\cdot,0)=w(\cdot)$ and 
	$$
	  \int_{\T}|\bar v(x,t)|^2+\tr\bar{R}(x,t)\,dx\leq \int_{\T}|w(x)|^2\,dx\quad\textrm{ for all }t>0,
	$$
	then $w$ is a wild initial datum in $C^\theta$. 
\end{corollary}

\begin{proof}
  Indeed, given a $C^\theta$-adapted subsolution	$(\bar{v},\bar{p},\bar{R})$, Proposition \ref{P:adaptsol} provides a sequence of $C^\theta$ weak solutions $(v_k,p_k)$ with $v_k(\cdot,0)=\bar{v}(\cdot, 0)$,
  $$
	  \int_{\T}|v_k(x,t)|^2\,dx=\int_{\T}|\bar v(x,t)|^2+\tr\bar{R}(x,t)\,dx\quad\textrm{ for all }t>0,
  $$
  and such that $v_k\to \bar{v}$ in $H^{-1}(\T)$ (uniformly in time).
\end{proof}

In the next section we show how these corollaries can be used to prove our main theorems. Then, in Sections \ref{s:strostri}-\ref{s:adaptsol} we give the proof of Propositions \ref{p:stristro}-\ref{P:adaptsol}.


\section{Proof of the main results}

Using Corollaries \ref{c:striadapt} and \ref{c:wild} from Section \ref{s:definitions}, we are in a position to prove the main theorems announced in the introduction.  

First of all we recall the following Lemma from \cite{SzekelyhidiJr:2014tu} (Lemma 12 on p238). By the remarks following Definition \ref{d:subsolution}, the concept of $L^\infty$-subsolution is substituted by that of strict subsolution.

\begin{lemma}\label{L:subsol1}
 Let $w\in L^2(\T;\R^3)$ with $\div\, w=0$. Then, for any $\eps>0$, there exists a smooth, strict subsolution $(\tilde v, \tilde p, \tilde R)$ on $[0,T]$ such that 
\begin{equation}
 \| \tilde v(0,\cdot)-w\|_{L^2}\leq\eps\label{E:es3}
\end{equation}
and
\begin{equation}\label{E:ew}
\int_{\T} |\tilde v(x,t)|^2+\tr\tilde R(x,t)\,dx\leq \int_{\T}|w(x)|^2\,dx+\eps\quad\textrm{ for all }t\in[0,T].
\end{equation}
\end{lemma}

With the help of this lemma we are now in a position to prove the density of wild initial data.

\begin{proof}[Proof of Theorem \ref{T:dense}]
Let $\delta>0$, $\theta<1/5$ and let $w\in L^2(\T;\R^3)$ with $\div\, w=0$. 
Using Lemma \ref{L:subsol1} we obtain a smooth, strict subsolution  $(\tilde v, \tilde p, \tilde R)$ on $[0,T]$ such that 
\eqref{E:es3}-\eqref{E:ew} hold for some $\eps>0$ to be fixed later. By adding a time-dependent non-negative multiple of the identity matrix to $\tilde R$ if necessary (which retains the property of being a strict and smooth subsolution), we may assume without loss of generality that 
in fact \eqref{E:ew} is an equality:
\begin{equation}\label{E:ew2}
\int_{\T} |\tilde v(x,t)|^2+\tr\tilde R(x,t)\,dx= \int_{\T}|w(x)|^2\,dx+\eps\quad\textrm{ for all }t\in[0,T].
\end{equation}

Let $\tilde v_0(\cdot)=\tilde v(\cdot,0)\in C^{\infty}(\T;\R^3)$ and note that
\begin{eqnarray*}
	\int_{\T}\tr\tilde R(x,0)\,dx&\leq& \|w\|_{L^2}^2-\|\tilde v_0\|_{L^2}^2+\eps\\
	&\leq& 2\eps\|\tilde v_0\|_{L^2}+\eps^2+\eps\\
	&\leq &2\eps\|w\|_{L^2}+3\eps^2+\eps\\
	&\leq &C\eps
\end{eqnarray*}
for some constant $C$ depending only on $w$. 

Next, apply Corollary \ref{c:striadapt} to obtain a $C^\theta$-adapted subsolution 
$(\bar v, \bar p, \bar R)$ with 
$$
\int_{\T} \bar v\otimes\bar v+\bar R\,dx=\int_{\T} \tilde v\otimes \tilde v+\tilde R\,dx\quad\textrm{ for all }t\in[0,T]
$$
and 
$$
\|\bar v-\tilde v\|_{H^{-1}}\leq \sigma
$$
for some $\sigma>0$ to be fixed later. Let $v_0(\cdot)=\bar{v}(\cdot,0)$ and note that
\begin{eqnarray*}
	\|v_0-\tilde v_0\|_{L^2}^2&=&\|v_0\|_{L^2}^2-\|\tilde v_0\|_{L^2}^2-2\int_{\T}\tilde v_0\cdot (v_0-\tilde v_0)\,dx\\
	&\leq& \int_{\T}\tr\tilde R(x,0)\,dx-2\int_{\T}\tilde v_0\cdot (v_0-\tilde v_0)\,dx\\
	&\leq &C\eps+2\sigma\|\tilde v_0\|_{H^1}.
\end{eqnarray*}
In particular, by choosing first $\eps>0$ sufficiently small and then $\sigma$ sufficiently small, we may ensure that
$$
\|v_0-\tilde v_0\|_{L^2}\leq \delta.
$$
Observing that for all $t\in[0,T]$
\begin{eqnarray*}
\int_{\T} |\bar v(x,t)|^2+\tr \bar R(x,t)\,dx&=&\int_{\T} |\tilde v(x,t)|^2+\tr \tilde R(x,t)\,dx\\
&=&\int_{\T}|w(x)|^2\,dx+\eps\\
&=&\int_{\T} |\bar v(x,0)|^2+\tr \bar R(x,0)\,dx\\
&=&\int_{\T} |v_0(x)|^2\,dx,	
\end{eqnarray*}
we see that, according to Corollary \ref{c:wild}, $v_0\in C^\theta(\T;\R^3)$ is a wild initial datum in $C^{\theta}$. This completes the proof.
\end{proof}

\bigskip

\begin{proof}[Proof of Theorem \ref{t:hprinciple}]
The proof is an immediate consequence of Corollary \ref{c:striadapt} and Proposition \ref{P:adaptsol}.	
\end{proof}

\section{Strong from strict subsolutions}\label{s:strostri}

\begin{proof}[Proof of Proposition \ref{p:stristro}]
The proof is a simplified version of the construction in \cite{Buckmaster:2014ty}, where we replace Beltrami flows by Mikado flows. Since the aim is only to produce a single reduction in the Reynolds tensor, there is no iteration here but just one single perturbation step.

\smallskip

Let 
$$
\delta_0=\inf\Bigl\{R(x,t)\xi\cdot\xi:\,|\xi|=1,\,x\in\T,t\in[0,T]\Bigr\}.
$$
Observe that since $R$ is a smooth positive definite tensor on the compact set $\T\times[0,T]$, $\delta_0>0$. Let $\delta<\delta_0$, so that $R-\delta\Id$ is positive definite on $[0,T]\times\T$.

\smallskip

We define $\bar v$ as 
$$
\bar v:=v+w=v+w_o+w_c,
$$
where the oscillation term $w_o$ and the corrector term $w_c$ are defined as follows. Let
$\Phi:\T\times [0,T]\to\T$ be the (periodic) flow of $v$ defined by
\begin{equation}
 \left\{\begin{aligned}
         \partial_t\Phi+v\cdot\nabla\Phi&=0\\
         \Phi(x,0)&=x,\qquad\forall\,x\in\T
        \end{aligned}\right.
\end{equation}
Set 
\begin{equation}\label{E:Rtilde}
 \tilde R(x,t)=D\Phi(x,t)\bigl(R(x,t)-\delta\Id\bigr)D\Phi^T(x,t). 
\end{equation}
Then $\tilde R$ is a positive-definite tensor on $\T\times[0,T]$, hence the set
$$
\mathcal{N}:=\left\{\tilde R(x,t):\,(x,t)\in \T\times [0,T]\right\}
$$
is a compact subset of $\SS^{3\times3}_+$ and by Lemma \ref{l:Mikado} there exists 
a smooth vectorfield $W:\mathcal{N}\times\T\to\R^3$ such that \eqref{e:Mikado}-\eqref{e:MikadoWW} hold.
We define
\begin{equation}
 w_o(x,t):=D\Phi(x,t)^{-1}W\left(\tilde R(x,t),\lambda\Phi(x,t)\right).
\end{equation}
Since the vector field $\xi\mapsto W(S,\xi)$ has zero average and is divergence-free, there exists 
$U=U(S,\xi)$ such that $\curl_\xi U=W$.
Then, for any\footnote{For the convenience of the reader we include the corresponding calculation in index notation, using the Einstein summation convention and the standard Kroenecker $\delta_{ij}$ and Levi-Civita alternating tensors $\varepsilon_{ijk}$: Using the identities $(\mathrm{cof}D\Phi)_{ij}=\frac{1}{2}\varepsilon_{ipq}\varepsilon_{jkl}\partial_k\Phi_p\partial_l\Phi_q$ and $\varepsilon_{jkl}\varepsilon_{jmn}=\delta_{km}\delta_{ln}-\delta_{kn}\delta_{lm}$, we obtain
\begin{align*}
\Bigl[\mathrm{cof}D\Phi^T(\curl U)(\Phi)\Bigr]_i&=\tfrac{1}{2}\varepsilon_{ikl}\varepsilon_{jpq}\partial_k\Phi_p\partial_l\Phi_q\varepsilon_{jmn}(\partial_mU_n)(\Phi)\\
&=\tfrac{1}{2}\varepsilon_{ikl}\partial_{k}\Phi_p\partial_l\Phi_q(\partial_pU_q-\partial_qU_p)(\Phi)\\
&=\varepsilon_{ikl}\partial_{k}\Phi_p\partial_l\Phi_q(\partial_pU_q)(\Phi)\\
&=\varepsilon_{ikl}\partial_k\Bigl(\partial_l\Phi_qU_q(\Phi)\Bigr)=\Bigl[\curl\bigl(D\Phi^TU(\Phi)\bigr)\Bigr]_{i}\,.
\end{align*}
} $S\in \SS^{3\times3}_+ $
\begin{equation*}
\begin{split}	
 \curl \left(D\Phi^TU(S,\lambda\Phi)\right)&=-\lambda\mathrm{cof}D\Phi^T\curl_{\xi}U(S,\lambda\Phi)\\
 &=-\lambda D\Phi^{-1}W(S,\lambda\Phi).
\end{split}
\end{equation*}
Therefore we set
\begin{equation}\label{e:Defw}
 w_c(x,t):=-\frac{1}{\lambda}\curl \left(D\Phi^T(x,t)U(\tilde R(x,t),\lambda\Phi(x,t))\right)-w_o(x,t),
\end{equation}
so that $\div w=0$.
Observe next that, as a consequence of the periodicity, the smoothness and \eqref{e:MikadoW}, $W$ can be written as
\begin{equation}\label{e:WFourier}
	W(S,\xi)=\sum_{k\in\Z^3\setminus\{0\}}a_k(S)B_ke^{ik\cdot \xi}, 
\end{equation}
for some complex vectors $B_k\in \C^3$ with $B_k\cdot k=0, |B_k|=1$, such that
$$
\sup_{S\in \mathcal N}|D_R^Na_k(S)|\leq \frac{C}{|k|^m}
$$
for any $m,N\in \N$ with a constant $C=C(\mathcal N, N,m)$. Therefore we can write
\begin{equation}\label{E:woexp}
 w_o(x,t)=\sum_{k\in\Z^3\setminus\{0\}}b_k(x,t)e^{ik\cdot\lambda\Phi(x,t)},
\end{equation}
where $b_k\in C^{\infty}(\T\times [0,T])$ with space-time $C^N$ norms bounded as
\begin{equation}\label{E:woexp2}
 \|b_k\|_{C^N(\T\times [0,T])}\leq \frac{C(N, R, v)}{|k|^4}.
\end{equation}
Furthermore, since $\det D\Phi\equiv 1$ ($v$ is divergence free), there exists a constant $c_0>1$ so that 
\begin{equation*}
 c_0^{-1}|k|\leq|D\Phi^T(x,t)k|\leq c_0|k|\quad\textrm{ for all }x\in\T,\,t\in[0,T].
\end{equation*}
In particular for any $k\in\Z^3$ the phase function 
$$
\varphi_k:=\frac{k}{|k|}\cdot\Phi
$$ 
satisfies the assumptions of Lemma \ref{L:stat} with constant $c_0$ independent of $k$.

Analogously, we have
\begin{equation}\label{E:wcexp}
 w_c(x,t)=\frac{1}{\lambda}\sum_{k\in\Z^3\setminus\{0\}}c_k(x,t)e^{ik\cdot\lambda\Phi(x,t)}
\end{equation}
with
\begin{equation}\label{E:wcexp2}
 \|c_k\|_{C^N(\T\times [0,T])}\leq \frac{C(N, R, v)}{|k|^4}.
\end{equation}
Indeed, from \eqref{e:WFourier} we deduce
$$
U(S,\xi)=\sum_{k\in\Z^3\setminus\{0\}}a_k(S)\frac{ik\times B_k}{|k|^2}e^{ik\cdot \xi}
$$
and therefore
$$
w_c=-\frac{i}{\lambda}\sum_{k\in\Z^3\setminus\{0\}}\nabla (a_k(\tilde R))\times \frac{D\Phi^T(k\times B_k)}{|k|^2}e^{i\lambda\Phi\cdot k}.
$$
In particular we see that $w=w_o+w_c$ is a smooth divergence-free vector field on $\T\times[0,T]$ such that
\begin{equation}\label{e:wowcC0}
\|w_o\|_0\leq C,\quad \|w_c\|_0\leq \frac{C}{\lambda}
\end{equation}
for some constant depending on $v,R$ but not on $\lambda$.

\bigskip

Set $\bar p=p$ and define $\bar R=\delta \Id+R_{11}+R_{12}$ with
\begin{equation}\label{e:defR11}
\begin{split}
R_{11}&:=-\mathcal{R}(F)-(w_c\otimes \bar{v}+w_o\otimes w_c),\\
F&:=\div(w_o\otimes w_o-R)+(\partial_t+v\cdot\nabla)w_o+w\cdot\nabla v+\partial_t w_c,\\
R_{12}&:=\fint_{\T}v\otimes v+R-\delta\Id -\bar v\otimes \bar v-R_{11}\,dx.
\end{split}
\end{equation}
Observe that $R_{12}$ is a function of $t$ only, hence using \eqref{e:Rproperty2}
\begin{align*}
-\div \bar{R}&=-\div R_{11}\\
&=F-\fint_{\T}F\,dx+\div(w_c\otimes \bar{v}+w_o\otimes w_c)\\
&=\partial_t\bar{v}+\div(\bar{v}\otimes\bar{v})+\nabla\bar{p}-\fint_{\T}F\,dx.
\end{align*}
Since $\div w=\div v=0$, we can write
\begin{equation*}
	F=\div(w_o\otimes w_o-R+w_o\otimes v+v\otimes w)+\partial_tw.
\end{equation*}
But then, since $v,R,w_o,w_c$ are periodic and --using \eqref{e:Defw}-- $w=\curl(z)$ for some periodic vectorfield $z$, it follows that 
$$
\int_{\T}F\,dx=0.
$$
We also easily see that for every $t$
$$
\int_{\T}\bar{v}\otimes\bar{v}+\bar{R}\,dx=\int_{\T}v\otimes v+R\,dx.
$$

\smallskip

Next, we will obtain estimates for the $H^{-1}$ norm of the perturbation $w$ and for the $C^0$ norm of the new Reynolds stress $\bar R$. They will turn out to be bounded by fixed constants --depending on the subsolution $(v,p,R)$ -- times a negative power of the parameter $\lambda$. Since this dependence is not relevant for our purposes, such constants will be denoted by the letter $C$, whose value may change from line to line. 

Moreover, in the estimates below we will repeatedly apply Lemma \ref{L:stat}, which requires the use of H\"older spaces $C^\alpha$. Therefore, for the sequel we fix $\alpha\in (0,1)$ and note that the estimates will depend also on $\alpha$. However, the precise choice of $\alpha$ is not important.

\smallskip
   
For any test function $f\in C^1(\T;\R^3)$ we have
\begin{align*}
 \Big|\int_{\T} w\cdot f\,dx\Big|&=\Big|\int_{\T} \sum_{k\in\Z^3\setminus\{0\}}(b_k+\lambda^{-1}c_k)\cdot fe^{i\lambda\Phi\cdot k}\,dx\Big|\\
&\leq\sum_{k\in\Z^3\setminus\{0\}}\Big|\int_{\T} (b_k+c_k)\cdot fe^{i\lambda|k|\varphi_k}\,dx\Big|\\
&\overset{\eqref{E:stati}}{\leq} C\sum_{k\in\Z^3\setminus\{0\}}\frac{(\|b_k\|_{1}+\|c_k\|_1))\|f\|_1}{\lambda|k|}\\
&\leq C\frac{\|f\|_1}{\lambda}\sum_{k\in\Z^3\setminus\{0\}}\frac{1}{|k|^5}\\
&\leq C\frac{\|f\|_1}{\lambda},
\end{align*}
hence
\begin{equation}
 \|\bar v-v\|_{H^{-1}}\leq\frac{C}{\lambda}.
\end{equation}
Now, let us proceed with the estimates for $\bar{R}$. To this end we write
\begin{equation}\label{e:defR11_1}
\begin{split}
R_{11}=-\mathcal{R}&\left(\div(w_o\otimes w_o-R)\right)-\mathcal{R}\left((\partial_t+v\cdot\nabla)w_o\right)\\
&-\mathcal{R}(w\cdot\nabla v)-\mathcal{R}(\partial_t w_c)-(w_c\otimes \bar{v}+w_o\otimes w_c),
\end{split}
\end{equation}
and we will estimate each term in the decomposition 
of $R_{11}$ in \eqref{e:defR11_1} successively.
By using the periodicity and \eqref{e:MikadoWW}, for any $S\in\SS^{3\times3}_+$ and $\xi\in\T$
$$
W(S,\xi)\otimes W(S,\xi)=S+\sum_{k\in\Z^3\setminus\{0\}}V_k(S)e^{i\xi\cdot k},
$$
where, as a consequence of \eqref{e:Mikado}, 
\begin{equation}\label{e:MikadoVk}
	V_k(S)k=0\quad\textrm{ for all $k\in\Z^3$ and $S\in\SS^{3\times3}_+$}.
\end{equation}
Hence
\begin{equation}\label{e:wowo}
\begin{split}
 w_o\otimes w_o&=R-\delta \Id+\sum_{k\in\Z^3\setminus \{0\}}D\Phi^{-1}V_k(\tilde R)D\Phi^{-T}e^{i\lambda\Phi\cdot k}\\
 &=R-\delta\Id  +\sum_{k\in\Z^3\setminus \{0\}}d_ke^{i\lambda\Phi\cdot k},
\end{split}
\end{equation}
where we have set $d_k=D\Phi^{-1}V_k(\tilde R)D\Phi^{-T}$. Observe that
\begin{equation*}
\begin{split}
\div\left(d_ke^{i\lambda\Phi\cdot k}\right)&=\div\left(d_k \right)e^{i\lambda\Phi\cdot k}+i\lambda D\Phi^{-1}V_k(\tilde R)D\Phi^{-T}D\Phi^T k e^{i\lambda\Phi\cdot k}\\
&\overset{\eqref{e:MikadoVk}}{=}\div \left(d_k\right)e^{i\lambda\Phi\cdot k}.
\end{split}
\end{equation*}
Therefore
\begin{equation}
 \div(w_o\otimes w_o-R)=\sum_{k\in\Z^3\setminus \{0\}}\div(d_k)e^{i\lambda|k|\varphi_k}\,.
\end{equation}
As before, since $W$ is smooth $d_k$ satisfies estimates of the form
\begin{equation*}
 \|d_k\|_{C^N(\T\times [0,T])}\leq \frac{C(N,R, v)}{|k|^4},
\end{equation*}
and therefore
\begin{equation}
\begin{split}
 \|\mathcal R(\div(w_o\otimes w_o-R))\|_{\alpha}& \leq \sum_{k\neq 0}\|\mathcal{R}(\div(d_k)e^{i\lambda|k|\varphi_k})\|_{\alpha}\\
& \overset{\eqref{E:statii}}{\leq} \frac{C}{\lambda^{1-\alpha}}\sum_{k\neq 0}|k|^{-5}\\
 &\leq \frac{C}{\lambda^{1-\alpha}}.
\end{split}
\end{equation}
Since $(\partial_t+v\cdot\nabla)\Phi=0$, the transport term
\begin{align}
 \partial_tw_o+v\cdot\nabla w_o=\sum_{k\in\Z^3\setminus\{0\}}(\partial_tb_k+v\cdot \nabla b_k)e^{i\lambda|k|\varphi_k}
\end{align}
 can again by \eqref{E:statii} be estimated as
\begin{equation}
 \|\mathcal R(\partial_tw_o+v\cdot\nabla w_o)\|_{\alpha}\leq \frac{C}{\lambda^{1-\alpha}}.
\end{equation}
Similarly we obtain
\begin{equation*}
\begin{split}
 \|\mathcal{R}(w\cdot\nabla v)\|_{\alpha}&\leq  \frac{C}{\lambda^{1-\alpha}}\,,\\
 \|\mathcal R(\partial_tw_c)\|_{\alpha}&\leq \frac{C}{\lambda^{2-\alpha}}\,,
\end{split}
\end{equation*}
and
\begin{equation*}
 \|w_c\otimes \bar v+w_o\otimes w_c\|_{0}\leq C\|w_c\|_0\leq \frac{C}{\lambda}.
\end{equation*}
In summary, we obtain 
\begin{equation}\label{E:R11est}
 \|R_{11}\|_0\leq \|R_{11}\|_\alpha\leq \frac{C}{\lambda^{1-\alpha}}.
\end{equation}
Concerning $R_{12}$, we calculate
\begin{align*}
	R_{12}&=\fint_{\T} R-\delta\Id-w_o\otimes w_o-w\otimes v-v\otimes w\,dx\\
	&-\fint_{\T}w_o\otimes w_c+w_c\otimes w_o+w_c\otimes w_c+R_{11}\,dx.
\end{align*}
Using \eqref{e:wowo} and Lemma \ref{L:stat} the first term can be estimated by $C\lambda^{-1}$, whereas for the second term we can use the $C^0$ estimates for $w_o$, $w_c$ and $R_{11}$ to again obtain the bound $C\lambda^{-1+\alpha}$. Consequently we obtain
\begin{equation}\label{E:R12est}
\|R_{12}\|_0\leq \frac{C}{\lambda^{1-\alpha}}.
\end{equation}
Finally, we turn to \eqref{e:H-1stristo2}. We have
\begin{align*}
  \bar{v}\otimes\bar v+\bar R-v\otimes v-R&=w_o\otimes w_o-(R-\delta\Id)+v\otimes w+w\otimes v+\\
  	&+w_o\otimes w_c+w_c\otimes w_o+w_c\otimes w_c+R_{11}+R_{12}.
\end{align*}
Using \eqref{E:woexp}, \eqref{E:wcexp}, \eqref{e:wowo} and \eqref{E:stati} we deduce that
for any $f\in C^1(\T;\R^3)$
$$
\left|\int_{\T}[w_o\otimes w_o-(R-\delta\Id)+v\otimes w+w\otimes v]f\,dx\right|\leq C\frac{\|f\|_{C^1}}{\lambda},
$$
whereas the $C^0$ estimates \eqref{e:wowcC0}, \eqref{E:R11est} and \eqref{E:R12est} easily imply 
$$
\left|\int_{\T}[w_o\otimes w_c+w_c\otimes w_o+w_c\otimes w_c+R_{11}+R_{12}]f\,dx\right|\leq C\frac{\|f\|_{C^0}}{\lambda^{1-\alpha}}.
$$
From this we can conclude that
$$
\|\bar{v}\otimes\bar v+\bar R-v\otimes v-R\|_{H^{-1}}\leq \frac{C}{\lambda^{1-\alpha}}\,.
$$

In summary, we have shown that $(\bar v,\bar p,\bar R)$ is a smooth subsolution satisfying \eqref{E:stristro2} such that
$$
\|v-\bar{v}\|_{H^{-1}}=O(\lambda^{-1}),
$$
and moreover
$$
\frac{1}{3}\tr\bar{R}=\delta + O(\lambda^{\alpha-1}),\quad \mathring{\bar{R}}=O(\lambda^{\alpha-1}).
$$
Thus, by choosing $\lambda$ sufficiently large we can ensure that $(\bar v,\bar p,\bar R)$ is a strong subsolution 
with \eqref{E:stristro3} such that also \eqref{E:H-1stristro}, \eqref{e:H-1stristo2} and \eqref{E:stristro1} are satisfied. This completes the proof.
\end{proof}


\section{Main perturbation step}\label{S:mainper}

In this section we recall the basic construction from \cite{Buckmaster:2014ty}. We state the result in a slightly more general form, more tailored for our purposes.

\begin{definition} Given $b>1$ we will call a sequence of numbers $(\delta_q,\lambda_q)$, $q\in\N$, with $\lambda_q\in\N$, {\it $b$-admissible} if the inequalities
\begin{equation}\label{e:b-admissible}
\delta_{q+1}\leq\frac{1}{2}\delta_q,\quad \lambda_q\leq \lambda_{q+1}^{\frac{2}{b+1}},\quad \delta_q^{\sfrac12}\lambda_q^{\sfrac15}\leq \delta_{q+1}^{\sfrac12}\lambda_{q+1}^{\sfrac15}
\end{equation}
are satisfied for any $q\in\N$.
\end{definition}

It is easy to see  (c.f.~\cite[Section 6]{Buckmaster:2014ty}) that if 
$$
\delta_q=a^{-b^q},\quad a^{cb^{q+1}}\leq \lambda_q \leq 2a^{cb^{q+1}},
$$
with $b>1$ and  $bc>5/2$, then $(\delta_q,\lambda_q)$ is $b$-admissible, provided $a\gg 1$ is sufficiently large (depending only on $b$ and $c$). 

We also recall from \cite{DeLellis:2013im,Buckmaster:2014ty} the following geometric lemma:
\begin{lemma}[Geometric Lemma]\label{l:split}
There exists $r_0>0$ and $\bar{\lambda} > 1$ with the following property.
There exist disjoint subsets $\Lambda_1,\Lambda_2\subset \{k\in \Z^3:\,|k|=\bar{\lambda}\}$
and smooth positive functions 
\[
\gamma^{(j)}_k\in C^{\infty}\left(B_{2r_0} (\Id)\right) \qquad j\in \{1,\dots, N\}, k\in\Lambda_j
\]
such that
\begin{itemize}
\item[(a)] $k\in \Lambda_j$ implies $-k\in \Lambda_j$ and $\gamma^{(j)}_k = \gamma^{(j)}_{-k}$;
\item[(b)] For each $R\in B_{2r_0} (\Id)$ we have the identity
\begin{equation}\label{e:split}
R = \frac{1}{2} \sum_{k\in\Lambda_j} \left(\gamma^{(j)}_k(R)\right)^2 \left(\Id - \frac{k}{|k|}\otimes \frac{k}{|k|}\right) 
\qquad \forall R\in B_{2r_0}(\Id)\, .
\end{equation}
\end{itemize}
\end{lemma}

Recall from Section \ref{s:definitions} that a (smooth) subsolution $(v,p,R)$ is a solution of the system
\begin{align*}
\partial_t v+\div (v\otimes v)+\nabla p&=-\div R\\
\div v&=0
\end{align*}
with $R(x,t)\geq 0$, and we say that the subsolution is {\it strong} if in addition \eqref{e:Rstr3} and \eqref{e:Rdeltaest} hold. This amounts to the condition that the tensor $R$ can be written as
\begin{equation}\label{e:Rdecompose} 
R(x,t)=\rho(t)\Id+\mathring{R}(x,t),
\end{equation}
with $\mathring{R}$ traceless and 
\begin{equation}\label{e:Rstrong}
\left|\mathring{R}(x,t)\right|\leq r_0\rho(t)\quad\textrm{ for all }(x,t).
\end{equation}

\begin{proposition}\label{p:basic}
Let $b>1$ and let $(\delta_q,\lambda_q)_{q\in\N}$ be a $b$-admissible sequence. 
Let $(v_0,p_0,R_0)$ be a smooth subsolution on $\T\times(T_1,T_2)$ and let 
$S\in C^{\infty}(\T\times(T_1,T_2);\S^{3\times 3})$ be a smooth matrix field such that 
$\tr S$ is a function of time only, and $\mathring{S}(x,t):=S(x,t)-\frac{1}{3}\tr S(t)\Id$ satisfies
\begin{equation}\label{e:basic:assumption1}
\left|\mathring{S}(x,t)\right|\leq \tfrac{1}{3}r_0\tr S(t)\quad\textrm{ for all }(x,t).
\end{equation}
Furthermore, let $M_0>0$ and $q\in \N$ be such that, for all $t\in (T_1,T_2)$
\begin{equation}\label{e:basic:assumption2}
\begin{split}
\|v_0\|_0&\leq M_0,\quad [v_0]_1\leq M_0\delta_q^{\sfrac12}\lambda_q,\\
[R_0]_1&\leq M_0\delta_{q+1}\lambda_q,\quad [p_0]_1\leq M_0^2\delta_q\lambda_q\,,
\end{split}
\end{equation}
\begin{equation}\label{e:basic:assumption3}
\frac{1}{3}|\tr S|\leq 4\delta_{q+1},
\end{equation}
and
\begin{equation}\label{e:basic:assumption4}
[S]_1\leq M_0\delta_{q+1}\lambda_q,\quad 
\|(\partial_t+v_0\cdot\nabla)S\|_0\leq M_0\delta_{q+1}\delta_q^{\sfrac12}\lambda_q.
\end{equation}

Then, for any $\eps>0$ there exists smooth $(v_1,p_1)\in C^{\infty}(\T\times(T_1,T_2);\R^3\times\R)$ and a smooth matrix field $\mathcal{E}\in C^{\infty}(\T\times(T_1,T_2);\S^{3\times 3})$ such that $\tr\mathcal{E}$ is a function of $t$ only, 
\begin{equation}\label{e:basic:average}
\int_{\T}v_1\otimes v_1\,dx=\int_{\T}v_0\otimes v_0+S+\mathcal{E}\,dx\quad\textrm{ for all $t$}
\end{equation}
and the equations
\begin{equation}\label{e:basic:equations}
\begin{split}
\partial_t v_1+\div (v_1\otimes v_1)+\nabla p_1&=-\div (R_0-S-\mathcal{E})\\
\div\, v_1&=0
\end{split}
\end{equation}
hold in $\T\times (T_1,T_2)$. Moreover, we have the estimates
\begin{eqnarray}
&\|v_1-v_0\|_{H^{-1}(\T)}\leq \overline{M}\delta_{q+1}^{\sfrac12}\lambda_{q+1}^{-1}\,,\label{e:basic:H-1}\\
&\|v_1-v_0\|_0\leq M\delta_{q+1}^{\sfrac12}\,,\quad[v_1-v_0]_1\leq M\delta_{q+1}^{\sfrac12}\lambda_{q+1}\,,\label{e:basic:v}\\
&\|p_1-p_0\|_0\leq M^2\delta_{q+1}\,,\quad [p_1-p_0]_1\leq M^2\delta_{q+1}\lambda_{q+1}\,,\label{e:basic:p}
\end{eqnarray}
with
$$
M=\overline{M}+C\lambda_{q+1}^{-\beta}\,.
$$
Furthermore, the error $\mathcal{E}$ satisfies 
\begin{equation}\label{e:basic:error}
\|\mathcal{E}\|_0+\frac{1}{\lambda_{q+1}}[\mathcal{E}]_1+\frac{1}{\delta_{q+1}^{\sfrac12}\lambda_{q+1}}\|(\partial_t+v_1\cdot\nabla)\mathcal{E}\|_0\leq C\delta_{q+1}^{\sfrac34}\delta_q^{\sfrac14}\lambda_q^{\sfrac12}\lambda_{q+1}^{\eps-\sfrac12}
\end{equation}
and similarly
\begin{equation}\label{e:basic:H-12}
\|v_1\otimes v_1-v_0\otimes v_0-S-\mathcal{E}\|_{H^{-1}(\T)}\leq C\delta_{q+1}^{\sfrac34}\delta_q^{\sfrac14}\lambda_q^{\sfrac12}\lambda_{q+1}^{\eps-\sfrac12}\,.
\end{equation}

In the above $C=C(b,M_0,\eps)$, $\beta=\beta(b)>0$ and $\overline{M}>1$ is a geometric constant. Finally, for times $t\notin\mathrm{supp}\,\tr S$ we have $v_1=v_0$, $p_1=p_0$ and $\mathcal{E}=0$.

\end{proposition}

The construction which lies at the heart of this proposition is precisely the construction used in \cite{Buckmaster:2014ty}, with $S=R_0$. In particular we do not claim any originality here, as  with minor modifications of \cite{Buckmaster:2014ty} one easily obtains Proposition \ref{p:basic}.

\begin{remark}
A remark concerning the constants in Proposition \ref{p:basic} is in order. Notice that in the assumptions \eqref{e:basic:assumption1}-\eqref{e:basic:assumption4} most estimates involve a constant $M_0$ (which in turn enters in the constant $C$ in the conclusions \eqref{e:basic:H-1}-\eqref{e:basic:error}), except for \eqref{e:basic:assumption3}, where we have written the constant $4$. The reason for this is that it is the quantity in \eqref{e:basic:assumption3} which solely determines the amplitude of the perturbation and hence results in the geometric constant $\overline{M}$. Of course more generally one could replace \eqref{e:basic:assumption3} by
$$
\frac{1}{3}|\tr S|\leq M_1\delta_{q+1},
$$
in which case $\overline{M}$ would depend on $M_1$ only. For our purposes this generalization is not useful and so we opted for introducing the minimal number of constants.
\end{remark}

\begin{proof}
Given $b>1$ and a $b$-admissible sequence $(\delta_q,\lambda_q)$, we set, following \cite[Section 6]{Buckmaster:2014ty},
\begin{equation}\label{e:def-muqlq}
\begin{split}
\mu=\mu_q&=\delta_{q+1}^{\sfrac14}\delta_q^{\sfrac14}\lambda_q^{\sfrac12}\lambda_{q+1}^{\sfrac12}\,,\\
\ell=\ell_q&=\delta_{q+1}^{-\sfrac18}\delta_q^{\sfrac18}\lambda_q^{-\sfrac14}\lambda_{q+1}^{-\sfrac34}\,.
\end{split}
\end{equation}
It can be verified directly by a short calculation (as in \cite[Section 6]{Buckmaster:2014ty}) that such a $b$-admissible sequence will then satisfy the conditions
\begin{equation}
\max\left\{\frac{\delta_q^{\sfrac12}\lambda_q\ell_q}{\delta_{q+1}^{\sfrac12}}\,;\frac{\delta_q^{\sfrac12}\lambda_q}{\mu_q}\,;\frac{1}{\ell_q\lambda_{q+1}}\,;\frac{\mu_q}{\delta_{q+1}^{\sfrac12}\lambda_{q+1}}\right\}\leq \lambda_{q+1}^{-\beta}
\end{equation}
for 
$$
\beta=\frac{1}{5}\frac{b-1}{b+1}>0.
$$
Consequently the conditions on the parameters from \cite[Section 2.6]{Buckmaster:2014ty} are satisfied and we may proceed as in \cite[Section 2]{Buckmaster:2014ty} to define $v:=\bar{v}+w$. We recall the main steps for the convenience of the reader. 

\bigskip

We fix a symmetric non-negative convolution kernel $\psi\in C^\infty_c (\R^3)$ and define
$v_\ell:=v_0*\psi_\ell$ and $S_\ell:=S*\psi_\ell$,
where the convolution is in the $x$ variables only. Next, we fix a smooth cut-off function $\chi\in C^\infty_c ((-\frac{3}{4}, \frac{3}{4}))$ such that 
$$
\sum_{l\in \Z} \chi^2 (x-l) = 1,
$$ 
and define $\Phi_l: \T\times [0,1]\to \T$ be the inverse flow of the periodic vector field $v_{\ell}$ starting at time $\tfrac{l}{\mu}$, i.e. the periodic solution of 
\begin{equation*}
\left\{\begin{array}{l}
\partial_t \Phi_l + v_{\ell}\cdot  \nabla \Phi_l =0\\ \\
\Phi_l (x,l \mu^{-1})=x\, .
\end{array}\right.
\end{equation*}
Set
$$
S_{\ell,l}(x,t)=S_{\ell}\left(\Phi_l(x,t),\tfrac{l}{\mu}\right),
$$
so that $S_{\ell,l}$ is the unique solution to the transport equation
\begin{equation*}
\left\{\begin{array}{l}
\partial_t S_{\ell,l} +v_\ell\cdot\nabla S_{\ell,l} = 0 \\
S_{\ell,l}(x,\frac l{\mu})=S_{\ell}(x,\frac{l}{\mu})\, .
\end{array}\right.
\end{equation*}
Note that, since by assumption $\tr S$ is a function of time only, 
$$
\rho_l:=\tfrac{1}{3}\tr S_{\ell,l}=\tfrac{1}{3}\tr S(\tfrac{l}{\mu})
$$ 
is a constant and moreover 
$$
\Bigl|\frac{S_{\ell,l}(x,t)}{\rho_l}-\Id\Bigr|\leq r_0.
$$
We next apply Lemma \ref{l:split}, denoting by $\Lambda^e$ and $\Lambda^o$ the corresponding families of frequencies in $\Z^3$, and set $\Lambda := \Lambda^o$ + $\Lambda^e$. For each $k\in \Lambda$ and each $l\in \Z\cap[0,\mu]$ we then set
\begin{align*}
\chi_l(t)&:=\chi\Bigl(\mu(t-l)\Bigr),\\
a_{kl}(x,t)&:=\sqrt{\rho_l}\gamma_k \left(\frac{S_{\ell,l}(x,t)}{\rho_l}\right),\\
w_{kl}(x,t)& := a_{kl}(x,t)\,B_ke^{i\lambda_{q+1}k\cdot \Phi_l(x,t)},
\end{align*}
where $B_k$ correspond to associated normalized Beltrami modes, i.e.~$B_k\in \C^3$ such that $|B_k|=1$, $B_{-k}=\overline{B_k}$ and 
$$
B_k\cdot k=0,\qquad k\times B_k=-i|k|B_k.
$$
The ``principal part'' of the perturbation $w$ consists of the map
\begin{align*}
w_o (x,t) := \sum_{\textrm{$l$ odd}, k\in \Lambda^o} \chi_l(t)w_{kl} (x,t) +
\sum_{\textrm{$l$ even}, k\in \Lambda^e} \chi_l(t)w_{kl} (x,t)\, 
\end{align*}
and the corrector $w_c$ is defined in such a way that $w:= w_o+w_c$ is divergence free:
\begin{equation*}
w_c:= \sum_{kl} \frac{\chi_l}{\lambda_{q+1}}\curl\left(ia_{kl}\phi_{kl}\frac{k\times B_k}{|k|^2}\right) e^{i\lambda_{q+1}k \cdot x}\,,
\end{equation*}
where $\phi_{kl}(x,t)=e^{i\lambda_{q+1}k\cdot[\Phi_l(x,t)-x]}$. As in Remark 1 in \cite{Buckmaster:2014ty}, we may write
$$
w=\sum_{kl}\chi_l\,L_{kl}\,e^{i\lambda_{q+1}k\cdot\Phi_l}\,.
$$
The new pressure is defined as 
\begin{equation*}
p_1:=p_0-\frac{|w_o|^2}{2} - \frac{1}{3} |w_c|^2 - \frac{2}{3} \langle w_o, w_c\rangle - \frac{2}{3} \langle v_0-v_\ell, w\rangle \,.
\end{equation*}
The new Reynolds stress term from \cite{Buckmaster:2014ty} will be
$$
\mathring{\mathcal{E}}^{(1)}:=\mathcal{R}\left[\partial_tv_1+\div(v_1\otimes v_1)+\nabla p_1+\div(R_0-S)\right],
$$
so that $\int_{\T}\mathring{\mathcal{E}}^{(1)}(x,t)\,dx=0$ for all $t$. Then, we define
$$
\mathcal{E}^{(2)}(t):=\fint_{\T}v_1\otimes v_1-v_0\otimes v_0-S\,dx
$$
and 
$$
\mathcal{E}(x,t):=\mathring{\mathcal{E}}^{(1)}(x,t)+\mathcal{E}^{(2)}(t).
$$
The assertions \eqref{e:basic:average} and \eqref{e:basic:equations} follow directly by construction. 

\smallskip

The estimates \eqref{e:basic:H-1}, \eqref{e:basic:v} and \eqref{e:basic:p} follow directly from the corresponding estimates in
\cite{Buckmaster:2014ty} and \eqref{e:basic:error} follows from the estimates for $\mathring{\mathcal{E}}^{(1)}$ in \cite{Buckmaster:2014ty} and the following two bounds for $\mathcal{E}^{(2)}$:
\begin{eqnarray}
\left|\int_{\T}v_1\otimes v_1-v_0\otimes v_0-S\,dx\right|&\leq& C\delta_{q+1}^{\sfrac34}\delta_q^{\sfrac14}\lambda_q^{\sfrac12}\lambda_{q+1}^{-\sfrac12},\label{e:basic:stressC0}\\
\left|\frac{d}{dt}\int_{\T}v_1\otimes v_1-v_0\otimes v_0-S\,dx\right|&\leq& C\delta_{q+1}^{\sfrac12}\delta_q\lambda_q,\label{e:basic:stressC1}
\end{eqnarray}
Indeed, we just need to check that the bound in \eqref{e:basic:stressC1} is better than the one claimed for $\|(\partial_t+v_1\cdot \nabla)\mathcal{E}\|_0$ in \eqref{e:basic:error}, i.e. that
$$
\delta_{q+1}^{\sfrac12}\delta_q\lambda_q\leq \delta_{q+1}^{\sfrac54}\delta_q^{\sfrac14}\lambda_q^{\sfrac12}\lambda_{q+1}^{\sfrac12}.
$$
This follows easily from \eqref{e:b-admissible}. 
 
\bigskip

{\bf Verification of \eqref{e:basic:stressC0}-\eqref{e:basic:stressC1}}

As in identity (81) in \cite{Buckmaster:2014ty}, we have
$$
w_o\otimes w_o-S=\sum_{k,l,k',l',k+k'\neq 0}\chi_{l}\chi_{l'}w_{kl}\otimes w_{k'l'}+\sum_l\chi_l^2[S_{\ell,l}-S_{\ell}]+[S_{\ell}-S].
$$
It then follows similarly to (81)-(82) in \cite{Buckmaster:2014ty} that
$$
\left|\int_{\T}w_o\otimes w_o-S\,dx\right|\leq C\frac{\delta_{q+1}\delta_q^{\sfrac12}\lambda_q}{\mu_q}+C\frac{\delta_{q+1}\lambda_q}{\lambda_{q+1}}.
$$
Using then the $H^{-1}$ estimate \eqref{e:basic:H-1}, the bounds on $\|w_o\|_0$ and $\|w_c\|_0$, and the expression for $\mu_q$ in \eqref{e:def-muqlq}, we arrive at 
\begin{equation*}
\begin{split}
\left|\int_{\T}v_1\otimes v_1-v_0\otimes v_0-S\,dx\right|&\leq C\frac{\delta_{q+1}^{\sfrac12}\delta_q^{\sfrac12}\lambda_q}{\lambda_{q+1}}+C\frac{\delta_{q+1}\delta_q^{\sfrac12}\lambda_q}{\mu_q}\\
&\leq C\delta_{q+1}^{\sfrac34}\delta_q^{\sfrac14}\lambda_q^{\sfrac12}\lambda_{q+1}^{-\sfrac12}
\left(1+\frac{\delta_q^{\sfrac14}\lambda_q^{\sfrac12}}{\delta_{q+1}^{\sfrac14}\lambda_{q+1}^{\sfrac12}}\right)\\
&\leq C\delta_{q+1}^{\sfrac34}\delta_q^{\sfrac14}\lambda_q^{\sfrac12}\lambda_{q+1}^{-\sfrac12}
\end{split}
\end{equation*}
thereby proving \eqref{e:basic:stressC0}. 

\smallskip

For evaluating the time derivative, observe first of all that, since $v_{\ell}$ is solenoidal, for any $F=F(x,t)$ 
\begin{equation}\label{e:basic:Dt}
\frac{d}{dt}\int_{\T}F\,dx=\int_{\T}D_tF\,dx,
\end{equation}
where $D_t=\partial_t+v_\ell\cdot\nabla$.
Indeed, if $X(x,t)$ denotes the flow associated to $v_{\ell}$, then $X(\cdot,t):\T\to\T$ is a diffeomorphism with $\det DX=1$ for all $t$, hence
$$
\int_{\T}F(x,t)\,dx=\int_{\T}F(X(x,t),t)\,dx.
$$
Differentiating in $t$ we arrive at \eqref{e:basic:Dt}. We apply this to 
\begin{equation*}
\begin{split}
&v_1\otimes v_1-v_0\otimes v_0-S=[w_o\otimes w_o-S]\\
&+[w_0\otimes w_c+w_c\otimes w_o+w_c\otimes w_c]+[v_0\otimes w+w\otimes v_0]
=F_1+F_2+F_3.
\end{split}
\end{equation*}
Since $D_tw_{kl}=0$, we have
$$
D_tF_1=\sum_{k,l,k',l',k+k'\neq 0}\chi_{l}'\chi_{l'}w_{kl}\otimes w_{k'l'}+D_t\sum_l\chi_l^2[S_{\ell,l}-S_{\ell}]+D_t[S_{\ell}-S].
$$
Using the argument of (81)-(82) as well as the estimates for $D_tR^5$ and $D_tR^4$ from \cite{Buckmaster:2014ty}, we obtain
$$
\left|\int_{\T}D_tF_1\,dx\right|\leq C\delta_{q+1}\delta_q^{\sfrac12}\lambda_q.
$$
Furthermore, using the estimate for $D_tR^2$ from \cite{Buckmaster:2014ty}, we deduce
$$
\|D_tF_2\|_0\leq C \delta_{q+1}\delta_q^{\sfrac12}\lambda_q.
$$
We turn to $D_tF_3$. 
$$
\left|\int_{\T}D_tw\otimes v_0\,dx\right|\leq \left|\int_{\T}D_tw\otimes v_{\ell}\,dx\right|+\left|\int_{\T}D_tw\otimes (v_0-v_{\ell})\,dx\right|,
$$
and for the second term above we can use the estimates for $\|D_tw\|_0$ and $\|v_0-v_{\ell}\|_0$ from 
 \cite{Buckmaster:2014ty} to conclude 
$$
\left|\int_{\T}D_tw\otimes (v_0-v_{\ell})\,dx\right|\leq C\delta_{q+1}^{\sfrac12}\delta_q^{\sfrac12}\lambda_q\mu_q\ell_q.
$$
Next, recall that we may write
$$
D_tw\otimes v_{\ell}=\sum_{kl}[\chi_lD_tL_{kl}+\chi_l'L_{kl}]\phi_{kl}\otimes v_\ell e^{i\lambda_{q+1}k\cdot x}=\sum_{kl}\tilde \Omega_{kl}(x,t)e^{i\lambda_{q+1}k\cdot x},
$$
and using the estimates from Lemmas 3.1 and 3.2 in  \cite{Buckmaster:2014ty}, for any $N\geq 1$
$$
[\tilde\Omega_{kl}]_N\leq C_N\delta_{q+1}^{\sfrac12}\mu_q\lambda_{q+1}^{N(1-\beta)}.
$$
Choosing $N\in\N$ so large that $N\beta\geq 1$ and using Proposition G.1 (i) in \cite{Buckmaster:2014ty}, we deduce 
$$
\left|\int_{\T}D_tw\otimes v_{\ell}\,dx\right|\leq C\delta_{q+1}^{\sfrac12}\mu_q\lambda_{q+1}^{-1}.
$$
On the other hand, 
\begin{equation*}
\begin{split}
D_tv_0&=\partial_tv_0+v_{\ell}\cdot\nabla v_0=\partial_tv_0+v_{0}\cdot\nabla v_0+(v_\ell-v_0)\cdot\nabla v_0\\
&=-\nabla p_0-\div R_0+(v_\ell-v_0)\cdot\nabla v_0,
\end{split}
\end{equation*}
and consequently $\|D_tv_0\|_0\leq 3M_0^2\delta_q\lambda_q$. It follows that
$$
\left|\int_{\T}w\otimes D_tv_0\,dx\right|\leq M\delta_{q+1}^{\sfrac12}\delta_q\lambda_q.
$$

Summarizing, we obtain
\begin{equation*}
\begin{split}
\biggl|\frac{d}{dt}\int_{\T}&v_1\otimes v_1-v_0\otimes v_0-S\,dx\biggr|\leq C\delta_{q+1}^{\sfrac12}\delta_q\lambda_q+C\frac{\delta_{q+1}^{\sfrac12}\mu_q}{\lambda_{q+1}}+
C\delta_{q+1}^{\sfrac12}\delta_q^{\sfrac12}\lambda_q\mu_q\ell_q\\
&=C\delta_{q+1}^{\sfrac12}\delta_q\lambda_q\left(1+\delta_{q+1}^{\sfrac14}\delta_q^{-\sfrac34}\lambda_q^{-\sfrac12}\lambda_{q+1}^{-\sfrac12}+\left(\frac{\delta_{q+1}^{\sfrac12}\lambda_q}{\delta_{q}^{\sfrac12}\lambda_{q+1}}\right)^{\sfrac14}\right)\\
&\leq C\delta_{q+1}^{\sfrac12}\delta_q\lambda_q\,.
\end{split}
\end{equation*}

Finally, the estimate \eqref{e:basic:H-12} is a consequence of the $C^0$ estimate of $\mathcal{E}$ as well as 
$$
\left|\int_{\T}[v_1\otimes v_1-v_0\otimes v_0-S]f\,dx\right|\leq C\|f\|_{C^1}\delta_{q+1}^{\sfrac34}\delta_q^{\sfrac14}\lambda_q^{\sfrac12}\lambda_{q+1}^{-\sfrac12}
$$ 
for any $f\in C^1(\T)$, whose proof is exactly as the proof of \eqref{e:basic:stressC0} above. 
\end{proof}


\section{Adapted subsolutions from strong subsolutions}\label{s:stroadapt}

In this section we show how to construct adapted subsolutions (c.f.~Definition \ref{D:adapt}) from strong subsolutions.

\begin{proof}[Proof of Proposition \ref{p:stradapt}]
We proceed by an iterative scheme based on Proposition \ref{p:basic}. 

\noindent{\bf Step 1. Definition of $(\delta_q,\lambda_q)$. }
We start by fixing various constants. Fix $\eps>0$ so that
$$
\frac{1}{5+2\eps}>\theta,
$$
and then choose $b,c>1$ so that
$$
\frac{1+4b}{2b}<c\textrm{ and }bc\leq 5/2+\eps.
$$
Next, let $\tilde\eps>0$ be sufficiently small, so that
$$
\tilde\eps b^2c<(b-1)\bigl[bc/2-b-1/4\bigr]
$$
(observe the the right hand side is positive because of the choice of $b,c$). 

\smallskip

Given $a\gg 1$ (to be chosen later), we then set, for $q=0,1,2,\dots$
\begin{equation}\label{e:adapted-deltalambda}
\delta_q:=\delta a^{b-b^q},\qquad \lambda_q\in [a^{cb^{q+1}},2a^{cb^{q+1}}]\cap\N\,.
\end{equation}
Finally, we fix $M_0>1$ so that 
$$
M_0\geq \max\{4\overline{M},\|v_0\|_0+4\delta^{\sfrac12}\},
$$ 
where $\overline{M}$ is the constant from Proposition \ref{p:basic}. Furthermore, let $C=C(b,M_0,\tilde\eps)$ and $\beta=\beta(b)$ as in Proposition \ref{p:basic}. 

It remains to choose the constant $a$. 
Recall that if $a\gg 1$ is sufficiently large, then the choice 
of $(\delta_q,\lambda_q)$ in \eqref{e:adapted-deltalambda} leads to a $b$-admissible sequence. Furthermore, we easily see that
$$
\delta_{q+2}^{-1}\delta_{q+1}^{\sfrac34}\delta_q^{\sfrac14}\lambda_q^{\sfrac12}\lambda_{q+1}^{\tilde\eps-\sfrac12}\leq \sqrt{2}a^{b^q[(b-1)((1-c/2)b+1/4)+\tilde\eps b^2c]},
$$
where the exponent is negative:
$$
(b-1)\left((1-c/2)b+1/4\right)+\tilde\eps b^2c<0
$$
by our choice of $b,c$ and $\tilde\eps$. Similarly
\begin{equation}\label{e:furtherconditions}
\frac{\delta_{q+1}\delta_q^{\sfrac12}\lambda_q}{\delta_{q+2}\delta_{q+1}^{\sfrac12}\lambda_{q+1}}\leq 2a^{b^q(b-1)(b+1/2-bc)}
\end{equation}
where once again the exponent is negative. 
Hence we can ensure, by choosing $a\gg 1$ sufficiently large, that 
\begin{equation}\label{e:constants1}
\begin{split}
C(b,M_0,\tilde\eps)\delta_{q+1}^{\sfrac34}\delta_q^{\sfrac14}\lambda_q^{\sfrac12}\lambda_{q+1}^{\tilde\eps-\sfrac12}&\leq\eta \delta_{q+2}\,,\\
C(b,M_0,\tilde\eps)\lambda_q^{-\beta}&\leq \overline{M}\,,\\
\delta_{q+1}\delta_q^{\sfrac12}\lambda_q&\leq \eta\, \delta_{q+2}\delta_{q+1}^{\sfrac12}\lambda_{q+1}\,\\
4\overline{M}\delta^{\sfrac12}\lambda_1^{-1}+2\delta_2&< \sigma
\end{split}
\end{equation} 
for all $q\in\N$, where $\eta\ll 1$ is a small constant to be specified later. Since $(v_0,p_0,R_0)\in C^{\infty}(\T\times [0,T])$, in the same way we can ensure additionally (by choosing $a\gg 1$ sufficiently large) that
\begin{equation}\label{e:startofiteration}
\begin{split}
[v_0]_1&\leq M_0\delta_0^{\sfrac12}\lambda_0\,,\\
[p_0]_1&\leq M_0^2\delta_0\lambda_0\,,\\
[R_0]_1&\leq M_0\delta_1\lambda_0\,,\\
\|(\partial_t+v_0\cdot\nabla)R_0\|_0&\leq \tfrac{1}{2}M_0\delta_1\delta_0^{\sfrac12}\lambda_0.
\end{split}
\end{equation}

Next, fix a cut-off function $\chi\in C^{\infty}[0,\infty)$ such that 
$$
\chi(t)=\begin{cases}1&t<\tfrac{1}{2}T,\\ 0 &t>\tfrac{3}{4}T,\end{cases}
$$
and set $\chi_q(t)=\chi(2^qt)$. As before, by choosing $a\gg 1$ sufficiently large, we may assume that
\begin{equation}\label{e:chiq}
|\chi'_q(t)|\leq \frac{1}{2}\delta_q^{\sfrac12}\lambda_q\qquad\textrm{ for all $q$.}
\end{equation}

\smallskip

\noindent{\bf Step 2. Inductive construction of $(v_q,p_q,R_q)$. }

Starting from $(v_0,p_0,R_0)$ and using Proposition \ref{p:basic} in $(0,T)$ we construct inductively a sequence $(v_q,p_q,R_q)$, $q\in\N$, of smooth strong subsolutions with
$$
R_q(x,t)=\rho_q(t)\Id+\mathring{R}_q(x,t)
$$
such that the following hold:
\begin{enumerate}
\item[($a_q$)] For all $t\in [0,T]$ we have
\begin{equation*}
\int_{\T}v_{q}\otimes v_{q}+R_{q}\,dx=\int_{\T}v_{0}\otimes v_{0}+R_{0}\,dx;
\end{equation*}
\item[($b_q$)] For all $t\in [0,T]$ 
\begin{equation}\label{e:RqC0-global}
\left|\mathring{R}_q(x,t)\right|\leq r_2\rho_q(t);
\end{equation}
\item[($c_q$)] If  $2^{-j}T<t\leq 2^{-j+1}T$ for some $j=1,\dots,q$, then 
\begin{equation}\label{e:rhoq-global}
\frac{3}{4}\delta_{j+1}\leq \rho_q(t)\leq \frac{3}{2}\delta_{j};
\end{equation}
\item[($d_q$)] For all $t\leq 2^{-q}T$
\begin{equation}\label{e:RqC0-local}
\left|\mathring{R}_q(x,t)\right|\leq r_3\rho_q(t),\quad \frac{3}{4}\delta_{q+1}\leq \rho_q(t)\leq \frac{5}{4}\delta_{q+1}\,;
\end{equation}
\item[($e_q$)] If $2^{-j}T<t\leq 2^{-j+1}T$ for some $j=1,\dots,q$, then
\begin{equation}\label{e:adapted-stepq}
\begin{split}
[v_q]_1&\leq M_0\delta_j^{\sfrac12}\lambda_j\,,\\
[p_q]_1&\leq M_0^2\delta_j\lambda_j\,,\\
[R_q]_1&\leq M_0\delta_{j+1}\lambda_j\,,\\
\|(\partial_t+v_q\cdot\nabla)R_q\|_0&\leq \tfrac{1}{2}M_0\delta_{j+1}\delta_j^{\sfrac12}\lambda_j,
\end{split}
\end{equation}
whereas, if $t\leq 2^{-q}T$, then \eqref{e:adapted-stepq} holds with $j=q$;
\item[($f_q$)] For all $t$
\begin{equation}\label{e:adapted-C0norm}
\|v_q\|_0\leq M_0\,.
\end{equation}
\end{enumerate}
Notice that $(v_0,p_0,R_0)$ satisfies $(a_0)$-$(f_0)$.
Suppose  $(v_q,p_q,R_q)$ is a smooth, strong subsolution on $[0,T]$ with 
$$
R_q(x,t)=\rho_q(t)\Id+\mathring{R}_q(x,t)
$$
such that the properties ($a_q$)-($f_q$) above hold. Define 
\begin{equation*}
\begin{split}
S_q(x,t)&:=\chi_q(t)\bigl[R_q(x,t)-\delta_{q+2}\Id\bigr]\\
&=\chi_q(t)\bigl[(\rho_q(t)-\delta_{q+2})\Id+\mathring{R}_q(x,t)\bigr].
\end{split}
\end{equation*}
Since $\supp\,\chi_q\subset[0,2^{-q}T)$, by ($d_q$) 
$$
\rho_q(t)\geq \frac{3}{4}\delta_{q+1}\geq \frac{3}{2}\delta_{q+2}\quad\textrm{ on }\supp\,\chi_q
$$
and therefore, using \eqref{E:r3r2r1},
$$
|\mathring S_q|=\chi_q|\mR_q|\leq \chi_qr_3\rho_q(t)\leq \chi_qr_0(\rho_q(t)-\delta_{q+2}).
$$
Consequently the tensor $S_q$ satisfies condition \eqref{e:Rstrong}. Moreover, since $|\chi_q(t)|\leq 1$,
$$
\frac{1}{3}\tr S_q(t)\leq \rho_q(t)\leq \frac{5}{4}\delta_{q+1},
$$
and $[S_q]_1\leq [R_q]_1$. Finally, using \eqref{e:chiq}, \eqref{e:RqC0-local}, \eqref{e:adapted-stepq} and \eqref{e:Rvsrho}
\begin{equation*}
\begin{split}
\|(\partial_t+v_q\cdot\nabla)S_q)\|_0&\leq \chi_q\|(\partial_t+v_q\cdot\nabla)R_q\|_0+|\chi_q'(t)|\|R_q\|_0\\
&\leq \tfrac{1}{2}M_0\delta_{q+1}\delta_q^{\sfrac12}\lambda_q+\delta_q^{\sfrac12}\lambda_{q}\rho_q(t)\mathbb 1_{\{\chi'_q\neq 0\}}\\
&\leq M_0\delta_{q+1}\delta_q^{\sfrac12}\lambda_q.
\end{split}
\end{equation*}
Hence the conditions of Proposition \ref{p:basic} are satisfied, and the proposition yields $(v_{q+1},p_{q+1})$ and $\mathcal{E}_{q+1}$ satisfying its conclusions \eqref{e:basic:average}-\eqref{e:basic:error}. Set
$$
R_{q+1}(x,t):=R_q(x,t)-S_{q}(x,t)-\mathcal{E}_{q+1}(x,t).
$$
We claim that $(v_{q+1},p_{q+1},R_{q+1})$ obtained in this way is a smooth strong subsolution satisfying ($a_{q+1}$)-($f_{q+1}$) above. First of all, it is clear by construction that 
\begin{equation}
(v_{q+1},p_{q+1},R_{q+1})=(v_q,p_q,R_q)\textrm{ for $t\geq 2^{-q}T$}
\end{equation}
and that, for all $t$
\begin{equation}
\int_{\T}v_{q+1}\otimes v_{q+1}+R_{q+1}\,dx=\int_{\T}v_{q}\otimes v_{q}+R_{q}\,dx
\end{equation}
and
\begin{equation*}
\begin{split}
\partial_tv_{q+1}+v_{q+1}\cdot\nabla v_{q+1}+p_{q+1}&=-\div R_{q+1},\\
\div v_{q+1}&=0.
\end{split}
\end{equation*}
In particular $(v_{q+1},p_{q+1},R_{q+1})$ is a smooth, strong subsolution of $[0,T]$ such that ($a_{q+1}$) holds, and in order to verify ($b_{q+1}$)-($f_{q+1}$) it suffices to check:
\begin{itemize}
\item \eqref{e:RqC0-global} and \eqref{e:rhoq-global} for $2^{-(q+1)}T\leq t\leq 2^{-q}T$ and $j=q+1$;
\item \eqref{e:RqC0-local} for $t\leq 2^{-(q+1)}T$;
\item \eqref{e:adapted-stepq} with $j=q+1$ for $t\leq 2^{-q}T$;
\item \eqref{e:adapted-C0norm} for $t\leq 2^{-q}T$.
\end{itemize}

From Proposition \ref{p:basic} and \eqref{e:constants1} we obtain
\begin{equation}
\|\mathcal{E}_{q+1}\|_0+\frac{1}{\lambda_{q+1}}[\mathcal{E}_{q+1}]_1+\frac{1}{\delta_{q+1}^{\sfrac12}\lambda_{q+1}}\|(\partial_t+v_q\cdot\nabla)\mathcal{E}_{q+1}\|_0\leq \eta \delta_{q+2}.
\end{equation}
Furthermore
$$
R_{q+1}(x,t)=\rho_{q+1}(t)\Id+\mathring{R}_{q+1}(x,t),
$$
where
\begin{equation*}
\begin{split}
\rho_{q+1}(t)&=(1-\chi_q(t))\rho_q(t)+\chi_q(t)\delta_{q+2}+\tilde\rho_{q+1}(t),\\
\mathring{R}_{q+1}(x,t)&=(1-\chi_q(t))\mathring{R}_q+\mathring{\mathcal{E}}_{q+1}(x,t),
\end{split}
\end{equation*}
and
$$
\tilde\rho_{q+1}(t)=\frac{1}{3}\tr\mathcal{E}_{q+1}(t)\quad\textrm{ and }\quad \mathcal{E}_{q+1}(x,t)=\tilde\rho_{q+1}(t)\Id+\mathring{\mathcal{E}}_{q+1}(x,t).
$$
Suppose $2^{-(q+1)}T\leq t\leq 2^{-q}T$. Then, using ($d_q$) 
$$
\bigl|\mathring{R}_{q+1}\bigr|\leq (1-\chi_q)\bigl|\mathring{R}_q\bigr|+\eta\delta_{q+2}\leq (1-\chi_q)r_3\rho_q+\eta\delta_{q+2}
$$
and
$$
\rho_{q+1}(t)\geq (1-\chi_q(t))\rho_q(t)+\chi_q(t)\delta_{q+2}-\eta\delta_{q+2}
$$
Since $0\leq \chi_q(t)\leq 1$, it follows 
\begin{equation}\label{e:basic:ineq1}
(1-\chi_q(t))r_3\rho_q(t)+\eta\delta_{q+2}\leq r_2[(1-\chi_q(t))\rho_q(t)+\chi_q(t)\delta_{q+2}-\eta\delta_{q+2}]\,,
\end{equation}
provided 
\begin{equation*}
\eta\leq \frac{r_2}{1+r_2}\quad\textrm{ and }\quad \eta\leq\frac32\frac{(r_2-r_3)}{1+r_2}.
\end{equation*}
Therefore, by choosing $\eta$ sufficiently small (depending only on $r_3,r_2$) we achieve \eqref{e:basic:ineq1}, from which
\eqref{e:RqC0-global} follows. Similarly, we estimate (recall, $2^{-(q+1)}T\leq t\leq 2^{-q}T$)
$$
(1-\eta)\delta_{q+2}\leq \rho_{q+1}(t)\leq \frac{5}{4}\delta_{q+1}+\eta\delta_{q+2}.
$$
By choosing $\eta<1/4$ we then achieve \eqref{e:rhoq-global} for $j=q+1$.

\smallskip

Next, let $t\leq 2^{-(q+1)}T$. Then $\chi_{q}(t)=1$, and hence $\rho_{q+1}=\delta_{q+2}+\tilde\rho_{q+1}$. We deduce
$$
\frac{3}{4}\delta_{q+2}\leq \rho_{q+1}\leq \frac{5}{4}\delta_{q+2}
$$
provided $\eta<1/4$. Since also $|\mathring{R}_{q+1}|\leq \eta\delta_{q+2}$, ($d_{q+1}$) follows by choosing $\eta$ sufficiently small (depending only on $r_3$). 

\smallskip

Now let us look at the estimates for $v_{q+1},p_{q+1}$ and $R_{q+1}$ for $t\leq 2^{-q}T$. Using 
\eqref{e:basic:v}, ($e_q$) and \eqref{e:constants1} we have
\begin{equation*}
\begin{split}
[v_{q+1}]_1&\leq [v_q]_1+[v_{q+1}-v_q]_1\leq M_0\delta_q^{\sfrac12}\lambda_q+\frac{M_0}{2}\delta_{q+1}^{\sfrac12}\lambda_{q+1}\\
&\leq M_0\delta_{q+1}^{\sfrac12}\lambda_{q+1}\,.
\end{split}
\end{equation*}
Similarly,
$$
[p_{q+1}]_1\leq M_0^2\delta_{q+1}\lambda_{q+1}\quad\textrm{ and }\quad [R_{q+1}]_1\leq M_0\delta_{q+2}\lambda_{q+1}
$$
using \eqref{e:basic:p}-\eqref{e:basic:error}, ($e_q$), \eqref{e:constants1} and \eqref{e:b-admissible}. Finally, using \eqref{e:basic:v}, \eqref{e:basic:p}, \eqref{e:adapted-stepq} and \eqref{e:constants1} 
\begin{equation*}
\begin{split}
\|(\partial_t+v_{q+1}&\cdot\nabla)R_{q+1}\|_0\leq \|(\partial_t+v_{q}\cdot\nabla)R_{q}\|_0+\|(\partial_t+v_{q}\cdot\nabla)S_q\|_0+\\
&+\|v_{q+1}-v_q\|_0\left([R_{q}]_1+[S_q]_1\right)+\|(\partial_t+v_{q+1}\cdot\nabla)\mathcal{E}_{q+1}\|_0\\
&\quad \leq 2M_0(1+2\overline{M})\delta_{q+1}\delta_q^{\sfrac12}\lambda_q+\eta\delta_{q+2}\delta_{q+1}^{\sfrac12}\lambda_{q+1}\\
&\quad \leq 3\eta M_0 (\overline{M}+1)\delta_{q+2}\delta_{q+1}^{\sfrac12}\lambda_{q+1}\\
&\quad \leq \tfrac{1}{2}M_0\delta_{q+2}\delta_{q+1}^{\sfrac12}\lambda_{q+1}
\end{split}
\end{equation*}
as required, provided $\eta$ is sufficiently small (depending only on $\overline{M}$). Therefore \eqref{e:adapted-stepq} holds with $j=q+1$. 

\smallskip

Finally, concerning the $C^0$ norm observe that the sequence $v_0,v_1,\dots,v_{q+1}$ that we defined inductively
also satisfies
\begin{equation}\label{e:c0estimates}
\begin{split}
\|v_{j+1}-v_j\|_{H^{-1}}\leq \overline{M}&\delta_{j+1}^{\sfrac12}\lambda_{j+1}^{-1},\quad \|v_{j+1}-v_j\|_0\leq 2\overline{M}\delta_{j+1}^{\sfrac12},\\ 
&\|p_{j+1}-p_j\|_0\leq 4\overline{M}^2\delta_{j+1}
\end{split}
\end{equation}
for all $j\leq q$. Moreover \eqref{e:b-admissible} implies $\delta_{q+1}\leq 2^{-q}\delta_1=2^{-q}\delta$, hence
$$
\sum_{q=0}^\infty\delta_{q+1}^{\sfrac12}\leq 4\delta^{\sfrac12}.
$$
It follows that 
$$
\|v_{q+1}\|_0\leq \|v_0\|_0+\sum_{j=0}^q\|v_{j+1}-v_j\|_0\leq \|v_0\|_0+4\delta^{\sfrac12}\leq M_0.
$$
This concludes the induction step.

\smallskip

\noindent{\bf Step 3. Convergence and conclusion.}

Overall we have shown that $(v_{q+1},p_{q+1},R_{q+1})$ satisfies ($a_{q+1}$)-($f_{q+1}$). 

The estimates \eqref{e:c0estimates} show that $\{v_q\}$ and $\{p_q\}$ are Cauchy sequences in $C^0$. 
Similarly, from the definition of $R_{q+1}$ and the inductive estimates we deduce
$$
\|R_{q+1}-R_{q}\|_0\leq \|S_q\|_0+\|\mathcal{E}_{q+1}\|_0\leq C\delta_{q+1},
$$
hence also $\{R_q\}$ is a Cauchy sequence. 

Furthermore, for each $t>0$ there exists $q_0=q_0(t)$ so that 
$$
(v_{q+1}(\cdot,t),p_{q+1}(\cdot,t),R_{q+1}(\cdot,t))=(v_{q}(\cdot,t),p_{q}(\cdot,t),R_{q}(\cdot,t))
$$ 
for all $q\geq q_0$. Consequently
$$
v_q\to\bar{v},\quad p_q\to\bar{p},\quad R_q\to\bar{R}\quad\textrm{ uniformly, }
$$
where $(\bar{v},\bar{p},\bar{R})\in C^{\infty}(\T\times (0,T])\cap C(\T\times [0,T])$ is a strong subsolution with
$$
|\mathring{\overline{R}}(x,t)|\leq r_2\overline{\rho}(t)
$$
such that
$$
\int_{\T}\bar{v}\otimes \bar{v}+\bar{R}\,dx=\int_{\T}v_0\otimes v_0+R_0\,dx\quad\textrm{ for all }t
$$
and, using once more \eqref{e:c0estimates} and \eqref{e:constants1},
\begin{equation*}
\begin{split}
\|\bar{v}-v_0\|_{H^{-1}}&\leq \sum_{q=0}^\infty \|v_{q+1}-v_q\|_{H^{-1}}\\
&\leq \sum_{q=0}^\infty \overline{M}\delta_{q+1}^{\sfrac12}\lambda_{q+1}^{-1}\leq 4\overline{M}\delta^{\sfrac	12}\lambda_1^{-1}< \sigma.
\end{split}
\end{equation*}
Similarly, recalling \eqref{e:basic:H-12} and \eqref{e:constants1} we obtain
\begin{equation*}
\begin{split}
\|\bar{v}\otimes\bar{v}-v_0\otimes v_0-R_0\|_{H^{-1}}&\leq \sum_{q=0}^\infty \|v_{q+1}\otimes v_{q+1}+R_{q+1}-v_q\otimes v_q-R_q\|_{H^{-1}}\\
&= \sum_{q=0}^\infty \|v_{q+1}\otimes v_{q+1}-v_q\otimes v_q-S_q-\mathcal{E}_{q+1}\|_{H^{-1}}\\
&\leq \sum_{q=0}^\infty \eta \delta_{q+2}\leq 2\delta_2< \sigma.
\end{split}
\end{equation*}

Concerning the initial datum, as a consequence of $(e_q)$ and \eqref{e:c0estimates} we have in particular
\begin{equation*}
\begin{split}
	[v_{q+1}(\cdot,0)-v_q(\cdot,0)]_{C^\theta}&\leq M_0\delta_{q+1}^{1/2}\lambda_{q+1}^{\theta}\\
	&\leq Ca^{b^{q+1}(\theta bc-1/2)}
\end{split}
\end{equation*}
for some constant $C$. 
By our choice of $\eps>0$ and $b,c>1$ we have
$$
bc\leq 5/2+\eps <\frac{1}{2\theta},
$$
hence the exponent in the above estimate is negative. Therefore, in the limit we have $\bar{v}(\cdot, 0)\in C^{\theta}(\T)$. Similarly we deduce  $\bar{p}(\cdot, 0)\in C^{2\theta}(\T)$ and, from $(d_q)$ we obtain $\bar{R}(\cdot,0)=0$. 

\medskip

It remains to verify conditions \eqref{e:adapted-est} for any $t>0$. To this end let $t\in [2^{-q}T,2^{-q+1}T]$ for some $q=0,1,2\dots.$. By our construction we have
$$
(\bar{v}(\cdot,t),\bar{p}(\cdot,t),\bar{R}(\cdot,t))=(v_{q}(\cdot,t),p_{q}(\cdot,t),R_{q}(\cdot,t)).
$$ 
Therefore, using $(c_q)$ and $(e_q)$, $\bar{\rho}(t)\leq 3/2\delta_q\leq Ca^{-b^q}$ and consequently
\begin{equation*}
\begin{split}
[\bar{v}(t)]_1&\leq Ca^{b^q(bc-1/2)}\leq M\bar{\rho}(t)^{-(bc-1/2)}\,,\\
[\bar{p}(t)]_1&\leq Ca^{b^q(bc-1)}\leq M\bar{\rho}(t)^{-(bc-1)}\,,\\
[\bar{R}(t)]_1&\leq Ca^{b^q(bc-b)}\leq M\bar{\rho}(t)^{-(bc-b)}\,,\\
\|(\partial_t+\bar{v}\cdot\nabla)\bar{R}(t)\|_0& \leq Ca^{b^q(bc-b-1/2)}\leq M\bar{\rho}(t)^{-(bc-b-1/2)},
\end{split}
\end{equation*}
for some constants $C,M$ depending only on $M_0$ and $a$. From here \eqref{e:adapted-est} follows by observing that, due to our choice of $b,c>1$, 
\begin{equation*}
\begin{split}
	bc-1/2&\leq 2+\eps,\, bc-1\leq 3/2+\eps\\
	bc-b&\leq 3/2+\eps,\, bc-b-1/2\leq 1+\eps.
\end{split}	
\end{equation*}
This concludes the proof.
\end{proof}

\section{Solutions from adapted subsolutions}\label{s:adaptsol}

In this section we show how to construct solutions from adapted subsolutions.

\begin{proof}[Proof of Proposition \ref{P:adaptsol}] 
As in the proof of Proposition \ref{p:stradapt}, we proceed by first defining an appropriate sequence $(\delta_q,\lambda_q)$.

\noindent{\bf Step 1. Definition of $(\delta_q,\lambda_q)$. }
Let $M>1$ and $\eps>0$ be the constants from Definition \ref{D:adapt}. Choose $b>1$ so that
$$
(2+\eps)b^2+\frac{1}{2}<\frac{1}{2\theta}.
$$
Since $\tfrac{5}{2}+\eps<\frac{1}{2\theta}$, such a choice is possible. Then, choose $c>1$ so that 
\begin{equation}\label{e:sol-bc}
(2+\eps)b^2+\frac{1}{2}<bc<\frac{1}{2\theta}.
\end{equation}
We note that, since $b>1$, \eqref{e:sol-bc} implies that 
$1+4b<2bc$. Next, let $\tilde\eps>0$ be sufficiently small so that
$$
\tilde\eps b^2c<(b-1)\bigl[bc/2-b-1/4\bigr],
$$
as in the proof of Proposition \ref{p:stradapt}. Given $a\gg 1$ (to be chosen later), we then set, for $q=0,1,2,\dots$
\begin{equation}\label{e:solution-deltalambda}
\delta_q:=\delta a^{b-b^q},\qquad \lambda_q\in [a^{cb^{q+1}},2a^{cb^{q+1}}]\cap\N\,,
\end{equation}
where
$$
\delta:=\max_{t\in[0,T]}\rho_0(t).
$$
As in the proof of Proposition \ref{p:stradapt}, the constant $a\gg 1$ will be chosen sufficiently large in such a way as to 
satisfy a number of criteria. However, before we discuss these criteria, we need to set $M_0$ in such a way that the following holds: 
if $\bar\rho(t)\geq \tfrac{3}{2}\delta_{q+2}$ for some $t>0$ and some $q=0,1,2,\dots$, then 
\begin{equation}\label{e:solution-adapted-stepq-2}
\begin{split}
	[\bar v(t)]_1&\leq M_0\delta_q^{\sfrac12}\lambda_q\,,\\
	[\bar p(t)]_1&\leq M_0^2\delta_q\lambda_q\,,\\
	[\bar R(t)]_1&\leq M_0\delta_{q+1}\lambda_q\,,\\
	\|(\partial_t+\bar v\cdot\nabla)\bar R(t)\|_0&\leq \tfrac{1}{16}M_0\delta_{q+1}\delta_q^{\sfrac12}\lambda_q\,.
\end{split}
\end{equation}
To show that such a choice of $M_0$ (which may depend on $M$ and $\delta$ but not on $a$)  is possible, note that, if 
$\bar\rho(t)\geq \tfrac32\delta_{q+2}$, then from \eqref{e:adapted-est} 
\begin{equation*}
	\begin{split}
		[\bar v(t)]_1&\leq M\bar\rho^{-(2+\eps)}\leq M (\delta a^b)^{-(2+\eps)}a^{(2+\eps)b^{q+2}}\\
			            &=M (\delta a^b)^{-(\sfrac52+\eps)} (\delta a^b)^{\sfrac12}a^{(2+\eps)b^{q+2}}\\
			            &\leq M\delta^{-(\sfrac52+\eps)} (\delta a^b)^{\sfrac12}a^{b^q(bc-\sfrac12)}\\
			            &\leq M\delta^{-(\sfrac52+\eps)}\delta_q^{\sfrac12}\lambda_q,
	\end{split}
\end{equation*}
where in the third line we have used \eqref{e:sol-bc} and that $a^b>1$.
Similar calculations involving the norms $[\bar p]_1$, $[\bar R]_1$ and $\|(\partial_t+\bar v\cdot\nabla)\bar R\|_0$ together with the inequalities
$$
(3/2+\eps)b^2\leq bc-b\leq bc-1\textrm{ and }(1+\eps)b^2\leq bc-b-1/2
$$
lead to the analogous conclusions. Accordingly, we fix $M_0>1$ so that \eqref{e:solution-adapted-stepq-2} holds and moreover
$$
M_0\geq \max\{4\overline{M},\|v_0\|_0+4\delta^{\sfrac12}\},
$$ 
where $\overline{M}$ is the constant from Proposition \ref{p:basic}. 
Finally, let $C=C(b,M_0,\tilde\eps)$ and $\beta=\beta(b)$ as in Proposition \ref{p:basic}. 

To choose $a\gg 1$, observe first of all that a sufficiently large choice guarantees that the choice 
of $(\delta_q,\lambda_q)$ in \eqref{e:solution-deltalambda} leads to a $b$-admissible sequence. Furthermore, by using the same calculations as in step 1 of the proof of Proposition \ref{p:stradapt}, we can ensure by choosing $a\gg 1$ sufficiently large, that 
\begin{equation}\label{e:constants2}
\begin{split}
C(b,M_0,\tilde\eps)\delta_{q+1}^{\sfrac34}\delta_q^{\sfrac14}\lambda_q^{\sfrac12}\lambda_{q+1}^{\tilde\eps-\sfrac12}&\leq\eta \delta_{q+2}\,,\\
C(b,M_0,\tilde\eps)\lambda_q^{-\beta}&\leq \overline{M}\,,\\
\delta_{q+1}\delta_q^{\sfrac12}\lambda_q&\leq \eta\, \delta_{q+2}\delta_{q+1}^{\sfrac12}\lambda_{q+1}\,\\
4\overline{M}\delta^{\sfrac12}\lambda_1^{-1}+2\delta_2&< \sigma\,\\
\delta_{q+1}&\leq \tfrac{1}{4}\delta_q
\end{split}
\end{equation} 
for all $q\in\N$, where $\eta\ll 1$ is a small constant to be specified later.
 
\bigskip

\noindent{\bf Step 2. Inductive construction of $(v_q,p_q,R_q)$. } 

\smallskip

Using Proposition \ref{p:basic} in $(0,T)$ we construct inductively the sequence $(v_q,p_q,R_q)\in C^{\infty}(\T\times(0,T])\cap C(\T\times[0,T])$ of strong subsolutions with 
$$
R_q(x,t)=\rho_q(t)\Id+\mathring{R}_q(x,t)
$$
such that the following hold:
\begin{enumerate}
\item[($a_q$)] For all $t\in [0,T]$
	\begin{equation}
		\int_{\T}v_{q}\otimes v_{q}+R_{q}\,dx=\int_{\T}v_{0}\otimes v_{0}+R_{0}\,dx;
	\end{equation}
\item[($b_q$)] For all $t\in [0,T]$ 
	\begin{equation}\label{e:RqC0-global-2}
		\left|\mathring{R}_q(x,t)\right|\leq r_1\rho_q(t);
	\end{equation}
\item[($c_q$)] For all $t\in [0,T]$
	\begin{equation}\label{e:rhoq-global-2}
		\rho_q(t)\leq 4\delta_{q+1};
	\end{equation}
\item[($d_q$)] If $\rho_q(t)\leq 2\delta_{q+2}$, then 
	\begin{equation}\label{e:RqC0-local-2}
		\left|\mathring{R}_q(x,t)\right|\leq r_2\rho_q(t);
	\end{equation}
\item[($e_q$)] If $\rho_q(t)\geq\tfrac{3}{2}\delta_{j+2}$ for some $j\geq q$, then
	\begin{equation}\label{e:adapted-stepq-2}
	\begin{split}
		[v_q(t)]_1&\leq M_0\delta_j^{\sfrac12}\lambda_j\,,\\
		[p_q(t)]_1&\leq M_0^2\delta_j\lambda_j\,,\\
		[R_q(t)]_1&\leq M_0\delta_{j+1}\lambda_j\,,\\
		\|(\partial_t+v_q\cdot\nabla)R_q(t)\|_0&\leq \tfrac{1}{16}M_0\delta_{j+1}\delta_j^{\sfrac12}\lambda_j;
	\end{split}
	\end{equation}
\item[($f_q$)] For all $t$
	\begin{equation}\label{e:adapted-C0norm-2}
		\|v_q\|_0\leq M_0\,.
	\end{equation}
\end{enumerate}
We set $(v_0,p_0,R_0):=(\bar v,\bar p,\bar R)$. Observe that because of Definition \ref{D:adapt} and due to our choice of $(\delta_q,\lambda_q) $ and $M_0$, $(v_0,p_0,R_0)$ satisfies the above assumptions $(a_0)$-$(f_0)$.

\bigskip
	
Suppose $(v_q,p_q,R_q)$ with 
$$
R_q(x,t)=\rho_q(t)\Id+\mathring{R}_q(x,t)
$$
satisfies ($a_q$)-($f_q$) above. Let 
$$
J_q:=\{t\in [0,T]:\,\rho_q(t)>\tfrac{3}{2}\delta_{q+2}\},\quad K_q:=\{t\in[0,T]:\,\rho_q(t)\geq 2\delta_{q+2}\}
$$
and let $\chi_q\in C_c^{\infty}(J_q)$ be such that 
$$
0\leq \chi_q(t)\leq 1\textrm{ for all $t$ and }\chi_q(t)=1\textrm{ for $t\in K_q$.}
$$
Observe that, if $t_0\in K_q$ and $t\in J_q$, then, using ($e_q$),  
$$
\rho_q(t)\geq \rho_q(t_0)-|t-t_0|\sup_{J_q}|\rho_q'|\geq 2\delta_{q+2}-\tfrac{1}{16}M_0\delta_{q+1}\delta_q^{\sfrac12}\lambda_q|t-t_0|.
$$
Consequently, for any $t_0\in K_q$ the open interval $(t_0-\tfrac{1}{\mu_q},t_0+\tfrac{1}{\mu_q})$ is contained in $J_q$, where
$$
\mu_q:=\frac{M_0}{8}\frac{\delta_{q+1}\delta_q^{\sfrac12}\lambda_q}{\delta_{q+2}}.
$$
Therefore we may choose the cut-off function $\chi_q$ in addition so that
\begin{equation}\label{e:chiq2}
|\chi_q'(t)|\leq \frac{3}{2}\mu_q=\frac{3}{16}M_0\frac{\delta_{q+1}}{\delta_{q+2}}\delta_q^{\sfrac12}\lambda_q.
\end{equation}
	
\smallskip
	
Define 
\begin{equation*}
\begin{split}
S_q(x,t)&:=\chi_q(t)\bigl[R_q(x,t)-\delta_{q+2}\Id\bigr]\\
&=\chi_q(t)\bigl[(\rho_q(t)-\delta_{q+2})\Id+\mathring{R}_q(x,t)\bigr].
\end{split}
\end{equation*}
By ($b_q$)
$$
|\mathring S_q|=\chi_q|\mR_q|\leq\chi_q r_1\rho_q,
$$
and since $r_1\leq\tfrac{1}{4}r_0$ and $\delta_{q+2}\leq \tfrac{2}{3}\rho_q(t)$ for $t\in \mathrm{supp }\chi_q\subset J_q$, 
$$
\chi_q r_1\rho_q\leq r_0\chi_q(\rho_q-\delta_{q+2}).
$$
Therefore the tensor $S_q$ satisfies condition \eqref{e:Rstrong}. Moreover, since $|\chi_q(t)|\leq 1$, using ($c_q$)
$$
\frac{1}{3}\tr S_q(t)\leq \rho_q(t)\leq 4\delta_{q+1},
$$
and $[S_q]_1\leq [R_q]_1$. Finally, using \eqref{e:chiq2}, \eqref{e:adapted-stepq-2} and  \eqref{e:Rvsrho} 
\begin{equation*}
	\begin{split}
	\|(\partial_t+v_q\cdot\nabla)S_q)\|_0&\leq \chi_q(t)\|(\partial_t+v_q\cdot\nabla)R_q\|_0+|\chi_q'(t)|\|R_q\|_0\\
	&\leq \frac{1}{16}M_0\delta_{q+1}\delta_q^{\sfrac12}\lambda_q+\frac{3}{8}M_0\frac{\delta_{q+1}}{\delta_{q+2}}\delta_q^{\sfrac12}\lambda_q\rho_q(t)\mathbb 1_{\{\chi'_q\neq 0\}}\\
	&\leq M_0\delta_{q+1}\delta_q^{\sfrac12}\lambda_q,
	\end{split}
\end{equation*}
where we have used in the last line that $\supp\,\chi_q'\subset J_q\setminus K_q$, so that $\rho_q(t)<2\delta_{q+2}$.
Hence the conditions of Proposition \ref{p:basic} are satisfied, and the proposition yields $(v_{q+1},p_{q+1})$ and $\mathcal{E}_{q+1}$ satisfying its conclusions \eqref{e:basic:average}-\eqref{e:basic:error}. Set
$$
R_{q+1}(x,t):=R_q(x,t)-S_{q}(x,t)-\mathcal{E}_{q+1}(x,t).
$$
We claim that $(v_{q+1},p_{q+1},R_{q+1})$ obtained in this way is a strong subsolution satisfying ($a_{q+1}$)-($f_{q+1}$) above. First
of all, it is clear by construction that 
\begin{equation}
	(v_{q+1},p_{q+1},R_{q+1})=(v_q,p_q,R_q)\textrm{ for $t\notin J_q$}
\end{equation}
and that, for all $t$
\begin{equation}
	\int_{\T}v_{q+1}\otimes v_{q+1}+R_{q+1}\,dx=\int_{\T}v_{q}\otimes v_{q}+R_{q}\,dx
\end{equation}
and
\begin{equation*}
\begin{split}
	\partial_tv_{q+1}+v_{q+1}\cdot\nabla v_{q+1}+p_{q+1}&=-\div R_{q+1},\\
	\div v_{q+1}&=0.
\end{split}
\end{equation*}
In particular $(v_{q+1},p_{q+1},R_{q+1})\in C^{\infty}(\T\times(0,T])\cap C(\T\times[0,T])$ is a strong subsolution such that ($a_{q+1}$) holds, and in order to verify ($b_{q+1}$)-($f_{q+1}$) it suffices to check:
\begin{itemize}
	\item \eqref{e:RqC0-global-2} and  \eqref{e:rhoq-global-2} for all $t\in J_q$;
	\item \eqref{e:RqC0-local-2} if $\rho_{q+1}(t)\leq 2\delta_{q+3}$;
	\item if $\rho_{q+1}(t)\geq\tfrac32\delta_{j+2}$ for some $j\geq q+1$, then \eqref{e:adapted-stepq-2} holds with $q$ replaced by $q+1$;
	\item \eqref{e:adapted-C0norm-2} for $v_{q+1}$ and all $t\in[0,T]$.
\end{itemize}
	
From Proposition \ref{p:basic} and \eqref{e:constants2} we obtain
\begin{equation*}
	\|\mathcal{E}_{q+1}\|_0+\frac{1}{\lambda_{q+1}}[\mathcal{E}_{q+1}]_1+\frac{1}{\delta_{q+1}^{\sfrac12}\lambda_{q+1}}\|(\partial_t+v_q\cdot\nabla)\mathcal{E}_{q+1}\|_0\leq \eta \delta_{q+2}.
\end{equation*}
Next, we have
$$
R_{q+1}(x,t)=\rho_{q+1}(t)\Id+\mathring{R}_{q+1}(x,t),
$$
where
\begin{equation*}
\begin{split}
	\rho_{q+1}(t)&=(1-\chi_q(t))\rho_q(t)+\chi_q(t)\delta_{q+2}+\tilde\rho_{q+1}(t),\\
	\mathring{R}_{q+1}(x,t)&=(1-\chi_q(t))\mathring{R}_q+\mathring{\mathcal{E}}_{q+1}(x,t),
\end{split}
\end{equation*}
and
$$
\tilde\rho_{q+1}(t)=\frac{1}{3}\tr\mathcal{E}_{q+1}(t)\quad\textrm{ and }\quad \mathcal{E}_{q+1}(x,t)=\tilde\rho_{q+1}(t)\Id+ \mathring{\mathcal{E}}_{q+1}(x,t).
$$
Then,
$$
\bigl|\mathring{R}_{q+1}(x,t)\bigr|\leq (1-\chi_q)\bigl|\mathring{R}_q(x,t)\bigr|+\eta\delta_{q+2}
$$
and
$$
\rho_{q+1}(t)\geq (1-\chi_q(t))\rho_q(t)+\chi_q(t)\delta_{q+2}-\eta\delta_{q+2}.
$$
If $t\in K_q$, then $\chi_q(t)=1$ and hence
$$
\bigl|\mathring{R}_{q+1}(x,t)\bigr|\leq \eta\delta_{q+2}\leq r_1(1-\eta)\delta_{q+2}\leq r_1\rho_{q+1}(t),
$$
provided $\eta>0$ is sufficiently small (depending on $r_1$). On the other hand, if $t\in J_q\setminus K_q$, then by $(d_q)$
\eqref{e:RqC0-local-2} holds and hence
$$
\bigl|\mathring{R}_{q+1}(x,t)\bigr|\leq r_2(1-\chi_q(t))\rho_q(t)+\eta\delta_{q+2}.
$$
Moreover, $\rho_q(t)>\tfrac{3}{2}\delta_{q+2}$ and  
\begin{equation}\label{e:adapt-sol-r2r1}
(1-\chi_q(t))r_2\rho_q(t)+\eta\delta_{q+2}\leq  r_1\bigl[(1-\chi_q(t))\rho_q(t)+\chi_q(t)\delta_{q+2}-\eta\delta_{q+2}\bigr]	
\end{equation}
provided
\begin{equation*}
\eta\leq \frac{r_1}{1+r_1}\quad\textrm{ and }\quad \eta\leq \frac{3}{2}\frac{r_1-r_2}{1+r_1}.
\end{equation*}
Therefore, by choosing $\eta$ sufficiently small (depending on $r_1,r_2$) we can ensure \eqref{e:adapt-sol-r2r1}, from which it follows that
$$
\bigl|\mathring{R}_{q+1}(x,t)\bigr|\leq r_1\rho_{q+1}(t).
$$
This concludes \eqref{e:RqC0-global-2}.
	
Similarly, if $t\in K_q$, we estimate
$$
\rho_{q+1}(t)\leq (1+\eta)\delta_{q+2}
$$
whereas, if $t\in J_q\setminus K_q$, then $\rho_q(t)\leq 2\delta_{q+2}$ and hence
$$
\rho_{q+1}(t)\leq  (1-\chi_q(t))\rho_q+\chi_q(t)\delta_{q+2}+\eta\delta_{q+2}\leq (2+\eta)\delta_{q+2}\leq 4\delta_{q+2},
$$
provided $\eta\leq 2$. Thus \eqref{e:rhoq-global-2} is proved as well.

Next, observe that 
\begin{eqnarray}
\rho_{q+1}(t)&\geq & (1-\chi_q(t))\rho_q(t)+\chi_q(t)\delta_{q+2}-\eta\delta_{q+2}\notag\\
&\geq &(1-\eta)\delta_{q+2}\geq \frac{3}{4}\delta_{q+2}\quad\textrm{ for }t\in J_q\label{e:sol-newrho},	
\end{eqnarray}
provided $\eta<1/4$. Since $\delta_{q+3}\leq \tfrac{1}{4}\delta_{q+2}$ (see \eqref{e:constants2}), it follows that $\rho_{q+1}(t)\geq 3\delta_{q+3}$ for all $t\in J_q$ and therefore $(d_{q+1})$ automatically follows from $(d_q)$. 

Now let us look at the estimates for $v_{q+1},p_{q+1}$ and $R_{q+1}$. From the above estimates we have that
$$
\rho_{q+1}(t)\geq \tfrac{3}{2}\delta_{q+3} \textrm{ for all }t\in J_q,
$$
whereas recall that $(v_{q+1},p_{q+1},R_{q+1})=(v_q,p_q,R_q)$ if $t\notin J_q$. Therefore it suffices to verify \eqref{e:adapted-stepq-2}
for $j=q+1$. 	
	
Using \eqref{e:basic:v}, ($e_q$) and \eqref{e:constants2} we have
\begin{equation*}
\begin{split}
	[v_{q+1}]_1&\leq [v_q]_1+[v_{q+1}-v_q]_1\leq M_0\delta_q^{\sfrac12}\lambda_q+\frac{M_0}{2}\delta_{q+1}^{\sfrac12}\lambda_{q+1}\\
	&\leq M_0\delta_{q+1}^{\sfrac12}\lambda_{q+1}\,.
\end{split}
\end{equation*}
Similarly,
$$
[p_{q+1}]_1\leq M_0^2\delta_{q+1}\lambda_{q+1}\quad\textrm{ and }\quad [R_{q+1}]_1\leq M_0\delta_{q+2}\lambda_{q+1}
$$
using \eqref{e:basic:p}-\eqref{e:basic:error}, ($e_q$), \eqref{e:constants2} and \eqref{e:b-admissible}. Finally, 
\begin{equation*}
\begin{split}
\|(\partial_t+v_{q+1}&\cdot\nabla)R_{q+1}\|_0\leq \|(\partial_t+v_{q}\cdot\nabla)R_{q}\|_0+\|(\partial_t+v_{q}\cdot\nabla)S_q\|_0+\\
&+\|v_{q+1}-v_q\|_0\left([R_{q}]_1+[S_q]_1\right)+\|(\partial_t+v_{q+1}\cdot\nabla)\mathcal{E}_{q+1}\|_0\\
&\quad \leq 2M_0(1+2\overline{M})\delta_{q+1}\delta_q^{\sfrac12}\lambda_q+\eta\delta_{q+2}\delta_{q+1}^{\sfrac12}\lambda_{q+1}\\
&\quad \leq 3\eta M_0 (\overline{M}+1)\delta_{q+2}\delta_{q+1}^{\sfrac12}\lambda_{q+1}\\
&\quad \leq \tfrac{1}{16}M_0\delta_{q+2}\delta_{q+1}^{\sfrac12}\lambda_{q+1}
\end{split}
\end{equation*}
as required, provided $\eta>0$ is sufficiently small (depending only on $\overline{M}$). This proves $(e_{q+1})$.

\bigskip

Finally, concerning the $C^0$ norm observe that the sequence $v_0,v_1,\dots,v_{q+1}$ that we defined inductively also satisfies
\begin{equation}\label{e:c0estimates2}
\begin{split}
\|v_{j+1}-v_j\|_{H^{-1}}\leq \overline{M}&\delta_{j+1}^{\sfrac12}\lambda_{j+1}^{-1},\quad \|v_{j+1}-v_j\|_0\leq 2\overline{M}\delta_{j+1}^{\sfrac12},\\ 
	&\|p_{j+1}-p_j\|_0\leq 4\overline{M}^2\delta_{j+1}
\end{split}
\end{equation}
for all $j\leq q$. Arguing as in the proof of Proposition \ref{p:stradapt} 
$$
\sum_{q=0}^\infty\delta_{q+1}^{\sfrac12}\leq 4\delta^{\sfrac12}
$$
and therefore
$$
\|v_{q+1}\|_0\leq \|v_0\|_0+\sum_{j=0}^q\|v_{j+1}-v_j\|_0\leq \|v_0\|_0+4\delta^{\sfrac12}\leq M_0.
$$
This concludes the induction step.
	
\bigskip

\noindent{\bf Step 3. Convergence and conclusion. } 

\smallskip
	
So far we have shown that $(v_{q+1},p_{q+1},R_{q+1})$ satisfies ($a_{q+1}$)-($f_{q+1}$). 
The estimates \eqref{e:c0estimates2} show that $\{v_q\}$ and $\{p_q\}$ are Cauchy sequences in $C^0$ and consequently
$$
v_q\to v,\quad p_q\to p\quad\textrm{ uniformly in $\T\times [0,T]$. }
$$
Moreover, $(b_q)$-$(c_q)$ imply that $R_q\to  0$ uniformly, hence $(v,p)$ is a weak solution of the Euler equations such that
$$
\int_{\T}v\otimes v\,dx=\int_{\T}v_0\otimes v_0+R_0\,dx\quad\textrm{ for all }t
$$
and, using once more \eqref{e:c0estimates2} and \eqref{e:constants2},
\begin{equation*}
\begin{split}
\|v-v_0\|_{H^{-1}}&\leq \sum_{q=0}^\infty \|v_{q+1}-v_q\|_{H^{-1}}\\
&\leq \sum_{q=0}^\infty \overline{M}\delta_{q+1}^{\sfrac12}\lambda_{q+1}^{-1}\leq 4\overline{M}\delta^{\sfrac	12}\lambda_1^{-1}< \sigma.
\end{split}
\end{equation*}
Furthermore, as in the proof of Proposition \ref{p:stradapt}, using \eqref{e:basic:H-12} and \eqref{e:constants2} we obtain
\begin{equation*}
\begin{split}
\|v\otimes v-v_0\otimes v_0-R_0\|_{H^{-1}}&\leq \sum_{q=0}^\infty \|v_{q+1}\otimes v_{q+1}+R_{q+1}-v_q\otimes v_q-R_q\|_{H^{-1}}\\
&\leq \sum_{q=0}^\infty \eta \delta_{q+2}\leq 2\delta_2< \sigma.
\end{split}
\end{equation*}

Concerning the initial datum, recall that in the construction above 
$$
(v_{q+1},p_{q+1},R_{q+1})=(v_{q},p_{q},R_{q})
$$
whenever $\rho_q(t)\leq \tfrac32 \delta_{q+2}$ and in particular also for $t=0$. Therefore 
$$
v(\cdot,0)=v_0(\cdot,0)\textrm{ and }p(\cdot,0)=p_0(\cdot,0).
$$
Finally let us turn to the H\"older-continuity of $v,p$. First of all recall from \eqref{e:sol-newrho} that along the iteration
$$
\rho_{q+1}(t)\begin{cases} \geq \tfrac{3}{4}\delta_{q+2}&\textrm{ if }t\in J_q,\\= \rho_q(t)&\textrm{ if }t\notin J_q.\end{cases}
$$
Consequently, if $q\geq j+1$ and $\rho_j(t)>\tfrac{3}{2}\delta_{q+2}$, then also $\rho_{j+1}(t)>\tfrac{3}{2}\delta_{q+2}$ (here we use \eqref{e:constants2}, in the form that $\delta_{j+2}>2\delta_{q+2}$). Hence, if $t\in (0,T]$ is such that $\rho_0(t)>\tfrac{3}{2}\delta_{q+2}$ for some $q$, then inductively we arrive at $\rho_q(t)>\tfrac{3}{2}\delta_{q+2}$. Conversely, it is easy to see that $\rho_{q+1}(t)\leq\rho_q(t)$ for all $q$ and $t$, hence we deduce
$$
J_q=\left\{t\in (0,T]:\,\rho_0(t)>\frac{3}{2}\delta_{q+2}\right\}.
$$
In particular $J_0\subset J_1\subset J_2\subset\dots$ is a nested sequence such that $\bigcup_qJ_q=(0,T]$.

Now let $t>0$ and let $\bar q=\bar q(t)$ be such that $t\in J_{\bar q}\setminus J_{\bar q-1}$. It follows that $t\in J_q$ for all $q\geq \bar q$ and in particular, from $(e_q)$ with $j=q$ we deduce 
$$
[v_q(t)]_1\leq M_0\delta_q^{\sfrac12}\lambda_q,\quad [p_q(t)]_1\leq M_0^2\delta_q\lambda_q
$$
for all $q\geq \bar{q}$. 
Recalling the $C^0$-estimates from \eqref{e:c0estimates2}, we easily deduce
\begin{equation*}
	[v_{q+1}(t)-v_q(t)]_1\leq 2M_0\delta_{q+1}^{\sfrac12}\lambda_{q+1},\quad \|v_{q+1}(t)-v_q(t)\|_0\leq 2M_0\delta_{q+1}^{\sfrac12}
\end{equation*}
hence by interpolation
\begin{equation*}
\begin{split}
	[v_{q+1}(t)-v_q(t)]_{C^\theta}&\leq 2M_0\delta_{q+1}^{1/2}\lambda_{q+1}^{\theta}\\
	&\leq Ca^{b^{q+1}(\theta bc-\sfrac12)}
\end{split}
\end{equation*}
for some constant $C$. 
By our choice of $b,c>1$ in \eqref{e:sol-bc} we have
$$
bc<\frac{1}{2\theta},
$$
hence the exponent in the above estimate is negative. Therefore, in the limit we have $v(\cdot, t)\in C^{\theta}(\T)$ with $C^{\theta}$-norm independent of $t$. This leads to \eqref{e:Ctheta-est}. In a similar manner we can deduce $p\in C([0,T];C^{2\theta}(\T))$. 
	
\end{proof}

\section{Acknowledgement}
Both authors gratefully acknowledge the support of Grant Agreement No.~277993 of the European Research Council. The authors would also like to thank the anonymous referee for the careful reading of the manuscript and for raising the issue about approximating arbitrary background flows in Section \ref{ss:approximation}.


\bibliographystyle{acm}

\end{document}